\documentclass[a4paper, 11pt, twoside]{amsart}

\usepackage[top=3cm, bottom=3cm, left=3cm, right=3cm]{geometry}
\usepackage{amsmath}
\usepackage{amsthm}
\usepackage{amssymb}
\usepackage{mathtools}
\usepackage[integrals]{wasysym}
\usepackage{stmaryrd}
\usepackage[mathcal]{euscript}

\usepackage[utf8]{inputenc}
\usepackage[T1]{fontenc}

\usepackage[english]{babel}

\usepackage{cite}
\usepackage{mathscinet}

\usepackage{lmodern}

\usepackage{url}
\usepackage[inline]{enumitem}
\usepackage{calc}

\usepackage[hidelinks, bookmarks, bookmarksnumbered, pdfstartview={XYZ null null 1.00}]{hyperref}

\theoremstyle{plain}
\newtheorem{theorem}{Theorem}[section]
\newtheorem{lemma}[theorem]{Lemma}
\newtheorem{corollary}[theorem]{Corollary}
\newtheorem{proposition}[theorem]{Proposition}
\newtheorem{claim}{Claim}
\theoremstyle{definition}
\newtheorem{definition}[theorem]{Definition}
\newtheorem{remark}[theorem]{Remark}

\def\N{\mathbb{N}}
\def\R{\mathbb{R}}
\def\e{\varepsilon}
\let\O=\Omega

\newcommand{\diverg}{\nabla\cdot}
\DeclareMathOperator{\diff}{d\!}

\DeclarePairedDelimiter{\norm}{\lVert}{\rVert}
\DeclarePairedDelimiter{\abs}{\lvert}{\rvert}

\newcommand{\suchthat}{\ifnum\currentgrouptype=16 \mathrel{}\middle|\mathrel{}\else\mid\fi}

\newcommand{\restrict}[1]{\bigr\rvert_{#1}}

\newcommand{\step}[2]{\medskip\noindent\emph{Step #1. #2.}}

\numberwithin{equation}{section}

\begin{document}

\setlength{\parskip}{1pt plus 1pt minus 1pt}

\title{Second order local minimal-time Mean Field Games}

\author{Romain Ducasse}
\address{Université de Paris and Sorbonne Université, CNRS, Laboratoire Jacques-Louis Lions (LJLL), F-75006 Paris, France.}
\email{ducasse@math.univ-paris-diderot.fr}

\author[Guilherme Mazanti]{Guilherme Mazanti$^*$}
\address{Uni\-ver\-sit\'e Paris-Saclay, CNRS, Cen\-trale\-Sup\-\'elec, Inria, La\-bo\-ra\-toire des signaux et sys\-t\`emes, 91190, Gif-sur-Yvette, France.}
\email{guilherme.mazanti@inria.fr}

\author{Filippo Santambrogio}
\address{Institut Camille Jordan, Universit\'e Claude Bernard - Lyon 1; 43 boulevard du 11 novembre 1918, 69622 Villeurbanne cedex, France \& Institut Universitaire de France.}
\email{santambrogio@math.univ-lyon1.fr}

\begin{abstract}
The paper considers a forward-backward system of parabolic PDEs arising in a Mean Field Game (MFG) model where every agent controls the drift of a trajectory subject to Brownian diffusion, trying to escape a given bounded domain $\Omega$ in minimal expected time. Agents are constrained by a bound on the drift depending on the density of other agents at their location. Existence for a finite time horizon $T$ is proven via a fixed point argument, but the natural setting for this problem is in infinite time horizon. Estimates are needed to treat the limit $T\to\infty$, and the asymptotic behavior of the solution obtained in this way is also studied. This passes through classical parabolic arguments and specific computations for MFGs. Both the Fokker--Planck equation on the density of agents and the Hamilton--Jacobi--Bellman equation on the value function display Dirichlet boundary conditions as a consequence of the fact that agents stop as soon as they reach $\partial\Omega$. The initial datum for the density is given, and the long-time limit of the value function is characterized as the solution of a stationary problem.
\end{abstract}


%35 Partial differential equations
%  35A General topics in partial differential equations
%    35A01 Existence problems for PDEs: global existence, local existence, non-existence
%  35B Qualitative properties of solutions to partial differential equations
%    35B40 Asymptotic behavior of solutions to PDEs
%  35D Generalized solutions to partial differential equations
%    35D30 Weak solutions to PDEs
%  35K Parabolic equations and parabolic systems
%    35K40 Second-order parabolic systems
%  35Q Partial differential equations of mathematical physics and other areas of application [See also 35J05, 35J10, 35K05, 35L05]
%    35Q89 PDEs in connection with mean field game theory

\maketitle

\newcommand\blfootnote[1]{%
  \begingroup
  \renewcommand\thefootnote{}\footnote{#1}%
  \addtocounter{footnote}{-1}%
  \endgroup
}

\blfootnote{$^*$Corresponding author}

\hypersetup{pdftitle={Second order local minimal-time Mean Field Games}, pdfauthor={Romain Ducasse, Guilherme Mazanti, Filippo Santambrogio}}

\tableofcontents

\section{Introduction}

Introduced around 2006 by Jean-Michel Lasry and Pierre-Louis Lions \cite{Lasry2006JeuxI, Lasry2006JeuxII, Lasry2007Mean} and at the same time by Peter Caines, Minyi Huang, and Roland Malham\'e \cite{Huang2003Individual, Huang2007Large, Huang2006Large}, the theory of Mean Field Games (MFGs, for short) describes the interaction of a continuum of players, assumed to be rational, indistinguishable, and negligible, when each one tries to solve a dynamical control problem influenced only by the average behavior of the other players (through a mean-field type interaction, using the physicists' terminology). The Nash equilibrium in these continuous games is described by a system of PDEs: a Hamilton--Jacobi--Bellman equation for the value function of the control problem of each player, where the distribution (density) of the players appears, coupled with a continuity equation describing the evolution of such a density, where the velocity field is the optimal one in order to solve the control problem, and is therefore related to the gradient of the value function. This system is typically forward-backward in nature: the density evolves forward in time starting from a given initial datum, and the value function backward in time, according to Bellman's dynamical programming principle, and its final value at a given time horizon $T$ is usually known.

The literature about MFG theory is quickly growing and many references are available. The 6-year course given by P.-L.~Lions at Collège de France, for which video-recording is available in French \cite{LLperso}, explains well the birth of the theory, but the reader can also refer to the lecture notes by P.~Cardaliaguet \cite{CardaliaguetNotes}, based on the same course.

In most of the MFG models studied so far the agents consider a fixed time interval $[0,T]$ and optimize a trajectory $x:[0,T]\to \Omega$ (where $\Omega\subset\R^d$ is the state space) trying to minimize a cost of the form 
$\int_0^T L(t,x(t),x^\prime(t),\rho_t)\diff t+\Psi(x(T),\rho_T),$
where $\rho_t$ denotes the distribution of players at time $t$. The function $L$ is typically increasing in $\abs{x^\prime}$ and, in some sense, in $\rho$. This means that high velocities are costly, and passing through areas where the population is strongly concentrated is also costly. Some MFGs, called \emph{MFGs of congestion} (see, for instance, \cite{AchPor-cong}), consider costs which include a product of the form $\rho_t(x(t))^\alpha\abs{x^\prime(t)}^\beta$ (for some exponents $\alpha, \beta > 0$), which means that high velocities are costly, and that they are even more costly in the presence of high concentrations. These models present harder mathematical difficulties compared to those where the cost is decomposed into $L(t,x(t),x^\prime(t))+g(t,x(t),\rho_t)$. Indeed, in many cases the latter MFG admits a variational formulation: equilibria can be found by minimizing a global energy among all possible evolutions $(\rho_t)_t$ (hence, they are {\it potential games}). This allows to prove the existence of the equilibrium via semicontinuity methods, and we refer to \cite{Benamou2017Variational} and \cite{SantambrogioCIME} for a detailed discussion of this branch of MFG theory.

When the MFG has no variational interpretation, then the existence of a solution is usually obtained via fixed-point theorems, but these theorems require much more regularity. Roughly speaking, given an evolution $\rho$ one computes the corresponding value function $\varphi$ as a solution to a Hamilton--Jacobi--Bellman equation and, given $\varphi$, one computes a new density evolution $\tilde\rho$ by following an evolution equation. We need existence, uniqueness, and stability results for these equations in order to find a fixed point $\tilde\rho=\rho$. This usually requires regularity of the velocity field $-\nabla\varphi$, which is difficult to prove, and can be essentially only obtained in two different frameworks: either the dependence of the cost functions on the distribution $\rho$ is highly regularizing (which usually means that it is non-local, and passes through averaged quantities such as convolutions $\int \eta(x-y)\diff\rho(y)$), or diffusion of the agents is taken into account, transforming the optimal control problem into a stochastic one. In this latter case, agents minimize $\mathbb E[\int_0^T L(t, X_t, \alpha(t), \rho_t) \diff t+\Psi(X_T,\rho_T)]$ where the process $X$ follows $\diff X_t=\alpha_t \diff t + \diff B_t$ and $(B_t)_{t \geq 0}$ represents a standard Brownian motion.

In \cite{Mazanti2019Minimal}, the second and third authors of the present paper introduced a different class of models, called \emph{minimal-time MFGs}. The main difference is that instead of considering a cost for the players penalizing both the velocity and the density, and minimizing the integral of such a cost on a fixed time interval $[0,T]$, the dynamics is subject to a constraint where the maximal velocity of the agents cannot exceed a quantity depending on the density $\rho_t$, and the goal of each agent is to arrive to a given target as soon as possible. In the typical situation, the target of the agents is the boundary $\partial\Omega$ of the domain where the evolution occurs. This can model, for instance, an evacuation phenomenon in crowd motion. The system that one obtains is the following
\begin{equation}
\label{MFGSyst-MS}
\left\{
\begin{aligned}
& \partial_t \rho - \diverg\left(\rho k[\rho] \frac{\nabla\varphi}{\abs{\nabla\varphi}}\right) = 0, & \quad & \text{ in } (0, T) \times \Omega, \\
& -\partial_t \varphi + k[\rho] \abs{\nabla\varphi} - 1 = 0, & & \text{ in } (0, T) \times \Omega, \\
& \rho(0, x) = \rho_0(x), & & \text{ in } \Omega, \\
& \varphi(t, x) = 0, & & \text{ on } (0, T) \times \partial\Omega,
\end{aligned}
\right.
\end{equation}
where the function $k[\rho_t](x)$ denotes the maximal speed that agents can have at point $x$ at time $t$, i.e., the dynamics is constrained to satisfy $\abs{x^\prime(t)}\leq k[\rho_t](x(t))$. 
Ideally, one would like to choose $k$ to be a non-increasing function of the density itself, such as $k[\rho](x)=(1-\rho(x))_+$. This choice is what is done in the well-known Hughes' model for crowd motion \cite{Hughes2002Continuum, Hughes2003Flow}. Indeed, this model is very similar to Hughes', which also considers agents who aim at leaving in minimal time a bounded domain under a congestion-dependent constraint on their speeds.

The main difference between the model in \cite{Mazanti2019Minimal} (from which the present paper stems) and Hughes' is that, in the latter, at each time, an agent moves in the optimal direction to the boundary assuming that the distribution of agents remains constant, whereas in \cite{Mazanti2019Minimal} and here agents take into account the future evolution of the distribution of agents in the computation of their optimal trajectories. This accounts for the time derivative in the Hamilton--Jacobi--Bellman equation from \eqref{MFGSyst-MS}, which is the main difference between \eqref{MFGSyst-MS} and the equations describing the motion of agents in Hughes' model and stands for the anticipation of future behavior of other agents. 

Another crucial (and disappointing) similarity between the above MFG system and Hughes' model is the fact that general mathematical results do not exist in the case $k[\rho]=(1-\rho)_+$ and more generally in the local case (except few results in the Hughes case in 1D). Indeed, the lack of regularity makes the model too hard to study, and the MFG case is not variational.

 In some sense the closest MFG model to this one is the one with multiplicative costs in \cite{AchPor-cong} (MFG with congestion). Indeed, an $L^\infty$ constraint $\abs{x^\prime}\leq k[\rho]$ can be seen as a limit as $m\to \infty$ of an integral penalization $$\int \abs*{\frac{\abs{x^\prime(t)}}{k[\rho_t](x(t))}}^m\diff t.$$
Note that the boundaries of the time interval have been omitted on purpose from the above integral, since the model in \cite{AchPor-cong} is set on a fixed time horizon but this is not part of our setting. For MFG with congestion, \cite{AchPor-cong} presents not only existence but also uniqueness results, under the assumption that the exponents appearing in the running cost satisfy a certain inequality. Unfortunately this inequality is never satisfied in the limit $m=+\infty$ as above, and it is not surprising that in our work we are not able to establish uniqueness results for our MFG system. Additional results for MFG with congestion were presented, for instance, in \cite{GomesVos}, under a smallness assumption on the time horizon, but this assumption cannot be made here, as the model is exactly meant to consider the case where a time horizon is not fixed. 

Because of these difficulties, \cite{Mazanti2019Minimal} studied the case of a non-local dependence of $k$ w.r.t.\ $\rho$ (say, $k[\rho](x)=\kappa(\int\eta(x-y)\diff\rho(y))$, for a non-increasing function $\kappa$ and a positive convolution kernel $\eta$), and proved existence of an equilibrium, characterized it as a solution of a non-local MFG system, and analyzed some examples, including numerical simulations. Instead, in the present paper we want to study the local case with diffusion. 

This means that we will consider a local dependence $k[\rho](x):=\kappa(\rho(x))$, and each agent solves a stochastic control problem
$$\inf\left\{\mathbb E[\tau]\,:\,X(\tau)\in\partial\Omega, X(0)=x_0,\,\diff X_t=\alpha_t \diff t+\sqrt{2\nu}\diff B_t,\,\abs{\alpha_t}\leq \kappa(\rho(t, X_t))\right\},$$
where $(B_t)_{t \geq 0}$ denotes a standard Brownian motion and the Brownian motions for all players are assumed to be mutually independent. Defining the corresponding value function $\varphi$, from classical results on stochastic optimal control (see \cite[Chapter IV]{Fleming2006Controlled}), under suitable assumptions, the optimal control is given in feedback form by
\[
\alpha_t = - \kappa(\rho(t, X_t))\frac{\nabla\varphi(t, X_t)}{\abs*{\nabla\varphi(t, X_t)}},
\]
(a definition which has to be carefully adapted to the case $\nabla\varphi=0$); moreover, the value function solves the Hamilton--Jacobi--Bellman equation
\[ -\partial_t \varphi(t, x) - \nu \Delta \varphi(t, x) + K(t, x) \abs*{\nabla\varphi(t, x)} - 1 = 0, \quad (t, x) \in [0, T) \times \Omega,\]
for $K=\kappa(\rho)$. Hence, we know the drift of the optimal stochastic processes followed by each agent, and this allows to write the Fokker--Planck equation solved by the law of this process. Putting together all this information, we obtain the following MFG system
\begin{equation}
\label{MFGSystInfiniteHorizon}
\left\{
\begin{aligned}
& \partial_t \rho - \nu \Delta \rho - \diverg\left(\rho \kappa(\rho) \frac{\nabla\varphi}{\abs{\nabla\varphi}}\right) = 0, & \quad & \text{ in } \mathbb R_+ \times \Omega, \\
& -\partial_t \varphi - \nu \Delta \varphi + \kappa(\rho) \abs{\nabla\varphi} - 1 = 0, & & \text{ in } \mathbb R_+ \times \Omega, \\
& \begin{aligned}
\rho(0, x) & = \rho_0(x), \\ 
\rho(t, x) & = 0, \quad \varphi(t, x) = 0,
\end{aligned} & & \begin{aligned}
& \text{ in } \O, \\
& \text{ on } \mathbb R_+ \times \partial\Omega,
\end{aligned}
\end{aligned}
\right.
\end{equation}
where $\Omega \subset \mathbb R^d$ is an open and bounded set, whose boundary will be supposed to be of class $C^2$ in this paper, $\nu > 0$ is a fixed constant, $\kappa: \mathbb R \to (0, +\infty)$, and $\rho_0\geq 0$ is the initial density. The Dirichlet condition on $\varphi$ comes as usual from the fact that, for agents who are already on the boundary, the remaining time to reach it is zero, and the Dirichlet condition on $\rho$ comes from the fact that we stop the evolution of a particle as soon as it touches the boundary (absorbing boundary conditions).

A crucial difference with the previous paper \cite{Mazanti2019Minimal} concerns the time horizon. If we suppose that $\kappa$ is bounded from below in the model without diffusion, it is not difficult to see that all agents will have left the domain after a common finite time, so that the final value of $\varphi$ is not really relevant, and the problem can be studied on a finite interval $[0,T]$. This is not the case when there is diffusion, as a density following a Fokker--Planck equation with a bounded drift cannot fully vanish in finite time. As a consequence, the model should be studied on the unbounded interval $[0,\infty)$. For every time $t<\infty$ there is still mass everywhere, but this mass decreases to $0$ as $t\to+\infty$, which suggests that the value function $\varphi$ should converge to a function, that we call $\Psi$, which is the value function for the corresponding control problem with no mass, i.e.\ when $\kappa=\kappa(0)$. Since in this control problem $\kappa$ is independent of time, $\Psi$ is a function of $x$ only and solves a stationary Hamilton--Jacobi--Bellman equation which takes the form of an elliptic PDE
$$-\nu \Delta\Psi+\kappa(0)\abs{\nabla\Psi}-1=0$$
with Dirichlet boundary conditions on $\partial\Omega$. It is then reasonable to investigate whether solutions of the above system satisfy further $\rho_t\to 0$ and $\varphi_t\to\Psi$ as $t\to+\infty$.

In order to study the above system, we will first study an artificial finite-horizon setting, where we stop the game at time $T$, choose a penalization $\psi:\Omega\to\R_+$ with $\psi=0$ on $\partial\Omega$, and look at the stochastic optimal control problem
\begin{multline*}
\inf\Big\{\mathbb E[\min\{\tau,T\}+\psi(X_{\min\{\tau,T\}})]\,: \\ \,X(\tau)\in\partial\Omega, X(0)=x_0,\,\diff X_t=\alpha_t \diff t+\sqrt{2\nu}\diff B_t,\,\abs{\alpha_t}\leq \kappa(\rho(t, X_t))\Big\}.
\end{multline*}
This gives rise to the MFG system
\begin{equation}
\label{MFGSystFiniteHorizon}
\left\{
\begin{aligned}
& \partial_t \rho - \nu \Delta \rho - \diverg\left(\rho \kappa(\rho) \frac{\nabla\varphi}{\abs{\nabla\varphi}}\right) = 0, & \quad & \text{ in } (0, T) \times \Omega, \\
& -\partial_t \varphi - \nu \Delta \varphi + \kappa(\rho) \abs{\nabla\varphi} - 1 = 0, & & \text{ in } (0, T) \times \Omega, \\
& \begin{aligned}
\rho(0, x) & = \rho_0(x), \quad & \varphi(T, x) & = \psi(x), \\ 
\rho(t, x) & = 0,  & \varphi(t, x) & = 0,
\end{aligned} & & \begin{aligned}
& \text{ in } \O, \\
& \text{ on } (0, T) \times \partial\Omega,
\end{aligned}
\end{aligned}
\right.
\end{equation}
which corresponds to \eqref{MFGSystInfiniteHorizon} with the unbounded time interval $\mathbb R_+$ replaced by $(0, T)$ and the additional final condition $\varphi(T,x)=\psi(x)$. We will prove the existence of a solution of the system for finite $T$ (note that for this system, as well as for its infinite-horizon counterpart, we are not able to prove uniqueness), and then consider the limit as $T\to\infty$. In order to guarantee suitable bounds, we just need to choose a sequence of final data $\psi_T$, possibly depending on $T$, which is uniformly bounded. We will then get at the limit a solution of the limit system which automatically satisfies $\rho_t\to 0$ (in the sense of uniform convergence) and $\varphi_t\to\Psi$ (this convergence being both uniform and strong in $H^1_0$). 

The paper is organized as follows. After this introduction, Section~\ref{SecPreliminary} presents the tools that we need to study the two separate equations appearing in System \eqref{MFGSystFiniteHorizon} on a finite horizon, which come from the classical theory of parabolic equations. Section~\ref{SecMFGFinite} is devoted to the existence of solutions of \eqref{MFGSystFiniteHorizon}. After providing a precise definition of solution of \eqref{MFGSystFiniteHorizon} taking care of the case $\nabla\varphi = 0$, we use the estimates of Section~\ref{SecPreliminary} to prove existence via a fixed-point argument based on Kakutani's theorem. Section~\ref{SecMFGInfinite} concerns the limit $T\to\infty$. In this section, some estimates of Section~\ref{SecPreliminary} need to be made more precise, in order to see how constants depend on the time horizon $T$. In this way we are able to prove existence of a limit of the solutions of \eqref{MFGSystFiniteHorizon} as the time horizon $T$ tends to $+\infty$ and that this limit solves the limit system \eqref{MFGSystInfiniteHorizon}. Then we consider the asymptotic behavior of a solution $(\rho, \varphi)$ of \eqref{MFGSystInfiniteHorizon} as $t \to +\infty$, proving first $\rho_t\to 0$ in $L^1$ and, thanks to a parabolic regularization argument, also in $L^\infty$.
To prove convergence in $L^1$, which is true for general Fokker--Planck systems under very mild assumptions, we exploit the MFG nature of the system, i.e.\ the coupling between the two equations, which also provides exponential decrease. We then consider the limit in time of $\varphi$, and prove that any bounded solution of this equation, once we know $\kappa(t,x)\to\kappa(0)$, can only converge as $t\to+\infty$ to the stationary function $\Psi$. This convergence is a priori very weak, but we are able to improve it into $L^\infty\cap H^1_0$, and to prove that the uniform convergences of both $\rho$ and $\varphi$ occur exponentially fast. The paper is then completed by an appendix, which details some global $L^\infty$ estimates for a large class of parabolic equations, including the estimates that we use to prove uniform convergence in time of $\rho_t$ and $\varphi_t$ to $0$ and $\Psi,$ respectively. These estimates are not surprising and not difficult to prove, using standard Moser iterations, but are not easy to find in the literature under the sole assumption of boundedness of the drift term in the divergence. The computations and the results are essentially the same as in the appendix of \cite{Cardaliaguet2013Long1}, but the boundary conditions are different. 

\section{Preliminary results}
\label{SecPreliminary}

This section presents some preliminary results on Fokker--Planck and Hamilton--Jacobi--Bellman equations which are useful for the analysis of the Mean Field Game systems \eqref{MFGSystInfiniteHorizon} and \eqref{MFGSystFiniteHorizon}. We recall that, in the whole paper, $\Omega$ denotes an open and bounded set whose boundary $\partial\Omega$ is assumed to be $C^2$. Even though some of the results presented in this preliminary section also hold without the smoothness assumption on $\partial\Omega$ (such as existence and uniqueness results for both Fokker--Planck and Hamilton--Jacobi--Bellman equations in Propositions~\ref{PropFP} and \ref{PropHJB}), this assumption is first used to obtain higher regularity of solutions of Hamilton--Jacobi--Bellman equations in Proposition~\ref{PropHJB} and is required for almost all of the subsequent results, including in particular our main results in Sections~\ref{SecMFGFinite} and \ref{SecMFGInfinite}, as a consequence of the need of higher regularity of $\varphi$.

\subsection{Fokker--Planck equation}

We recall some results on the Fokker--Planck equation on a bounded domain $\O\subset\R^d$ in finite time horizon $T \in (0, +\infty)$,
\begin{equation}
\label{FPEqn}
\left\{
\begin{aligned}
& \partial_t \rho - \nu \Delta \rho + \diverg (\rho V) = 0 & \quad & \text{ in } (0, T) \times \Omega, \\
& \rho(0, x) = \rho_0(x) & & \text{ in } \Omega, \\
& \rho(t, x) = 0 & & \text{ in } [0, T] \times \partial\Omega,
\end{aligned}
\right.
\end{equation}
where $V: (0, T) \times \Omega \to \mathbb R^d$ is a given velocity field. We will only focus on the case where $V$ is bounded, an assumption which is satisfied in the cases of interest for this paper and which strongly simplifies the analysis. The results presented in this short section are a mixture of classical results (for which we mainly refer to \cite[Section~7.1]{Evans1998Partial} or \cite[Chapter~III]{Ladyzhenskaja1968Linear}), recent results obtained by Porretta in \cite{Porretta2015Weak}, and extra computations which are not original but are difficult to find in the literature, which we present in the Appendix. 

 \begin{definition}
 \label{DefiWeakSolFP}
Let $\nu > 0$, $V \in L^\infty((0, T) \times \Omega ; \mathbb R^d)$, and $\rho_0 \in L^1(\Omega)$. We say that $\rho \in L^1((0, T)\times\Omega)$ is a \emph{weak solution} of \eqref{FPEqn} if, for every $\eta \in C^2([0, T] \times \Omega)$ such that $\eta\restrict{[0, T] \times \partial\Omega} = 0$ and $\eta\restrict{\{T\} \times \Omega} = 0$, one has
 \begin{equation}
\label{FPWeak}
 -\int_0^T \int_\Omega \rho \partial_t \eta \diff x \diff t - \int_0^T \int_{\Omega} \left(\nu \rho \Delta\eta + \rho V\cdot \nabla\eta\right) \diff x \diff t = \int_\Omega \rho_0(x) \eta(0, x) \diff x.
 \end{equation}
 \end{definition}

We observe that, whenever equality \eqref{FPWeak} holds for $C^2$ functions, and if we have further that $\rho\in L^2((t_1,t_2);H^1_0(\Omega))\cap C^0([t_1,t_2];L^2(\O))$ for some $t_1, t_2 \in (0, T)$ with $t_1 < t_2$, then we also have
\begin{multline}
\label{FPWeakH1_2}
\int_{t_1}^{t_2} \left(\int_\Omega -\rho \partial_t \eta + \nu \nabla\rho\cdot \nabla\eta - \rho V\cdot \nabla\eta\right) \diff x \diff t  \\ = \int_\Omega \rho(t_1,x) \eta(t_1, x) \diff x-\int_\Omega \rho(t_2,x) \eta(t_2, x) \diff x.
\end{multline}
for every $\eta\in C^1_c([0,T)\times\Omega))$ and, by density, for every $\eta\in L^2((t_1, t_2);H^1_0(\Omega)) \cap C^0([t_1, t_2];\allowbreak L^2(\Omega))$ such that $\partial_t \eta \in L^2((t_1, t_2); H^{-1}(\Omega))$. Of course it is well-known that, in case $\rho$ is more regular, other test functions can also be accepted, and that if $\rho\in C^2$ then the equation is satisfied in a classical sense.
 
We now state a proposition summarizing all the main results that we will use.

\begin{proposition}
\label{PropFP}
Let $\nu > 0$, $V \in L^\infty((0, T) \times \Omega; \mathbb R^d)$, and $\rho_0 \in L^1(\Omega)$ be a given non-negative initial datum. Then \eqref{FPEqn} admits a unique weak solution $\rho$. In addition, we have $\rho\geq 0$ and $\rho\in C^0([0,T];L^1(\Omega))$ with $\norm{\rho_t}_{L^1}\leq \norm{\rho_0}_{L^1}$, as well as $\nabla \rho\in L^q((0, T) \times \Omega)$ and $\partial_t\rho\in L^q((0, T); W^{-1,q}(\Omega))$ for all $q<\frac{d+2}{d+1}$ and $\rho\in L^r((0,T)\times\Omega)$ for all $r<\frac{d+2}{d}$, and the norms of $\rho, \nabla \rho$ and $\partial_t\rho$ in the above spaces are bounded by quantities only depending on $\norm{\rho_0}_{L^1}$.
Moreover, for every $t_0>0$, we also have $\rho\in L^\infty((t_0,T)\times\Omega)\cap L^2((t_0,T);H^1_0(\Omega))\cap C^0([t_0,T];L^2(\O))$ and $\partial_t\rho\in L^2((t_0,T);H^{-1}(\Omega)) $.
\end{proposition}

Of course we do not provide a full proof of the above results, but we explain below how to deduce the different parts of the statement from the most well-known literature and the relevant references.

\begin{proof}
The definition of the solution is exactly the one used in \cite{Porretta2015Weak}, where the key assumption is $\rho|V|^2\in L^1((0,T)\times\Omega)$. In our case, where $V$ is bounded, this assumption is satisfied as soon as $\rho\in L^1((0,T)\times\Omega)$. One of the main results of \cite{Porretta2015Weak} is exactly the uniqueness of the solution in this class, and this can be applied to the present setting. The same paper also guarantees the estimates $\nabla \rho\in L^q((0, T) \times \Omega)$, $\partial_t\rho\in L^q((0, T); W^{-1,q}(\Omega))$, $\rho\in L^r((0,T)\times\Omega)$, and the $L^1$ bound.

Existence is not included in \cite{Porretta2015Weak} but in the particular case $V\in L^\infty$ it is easy to obtain by regularization and compactness. Indeed, one can apply the classical $L^2$ theory of \cite[Chapter III]{Ladyzhenskaja1968Linear} to an approximated initial datum, and obtain a sequence of solutions: the $L^r$ bounds of \cite{Porretta2015Weak}, which only depend on the initial $L^1$ norm in this setting, allow to obtain the compactness we need to pass the PDE to the limit. Note that this argument is specific to the case $V\in L^\infty$ since, otherwise, we would need to control the $L^1$ norm of $\rho|V|^2$, which is non-trivial.

By approximating $\rho_0$ and $V$ by smooth functions $\rho_{0, \varepsilon}$ and $V_\varepsilon$ with $\rho_{0, \varepsilon} \geq 0$, the corresponding solution $\rho_\varepsilon$ of \eqref{FPEqn} satisfies $\rho_\varepsilon \geq 0$ thanks to the classical maximum principle, and then this property passes to the limit and also applies to the unique weak solution $\rho$ corresponding to the original $\rho_0$ and $V$.

The local $L^\infty$ bound can be obtained thanks to the Appendix of the present paper (even if we stress that similar computations are nowadays standard). For simplicity, the bound is presented under the assumption $\rho_0\in L^r$, $r>1$, and not $\rho_0\in L^1$. Yet, the time-space $L^r$ summability already stated in the claim allows to deduce $\rho_t\in L^r$ for a.e. $t>0$, and if we choose $t<t_0$ we obtain the desired $L^\infty$ bound. Once we know that $\rho$ is locally (in time) $L^\infty$ (in space), it is also locally (in time) $L^2$ (in space), and hence the classical $L^2$ theory of \cite[Chapter~III]{Ladyzhenskaja1968Linear} provides the last estimates of the statement.
\end{proof}

\subsection{Hamilton--Jacobi--Bellman equation}

We consider the non-linear Hamilton--Ja\-co\-bi--Bellman equation in a finite time horizon
\begin{equation}
\label{HJBEqn}
\left\{
\begin{aligned}
& -\partial_t \varphi - \nu \Delta \varphi + K \abs{\nabla\varphi} - 1 = 0 & \quad & \text{ in } (0, T) \times \Omega, \\
& \varphi(T, x) = \psi(x) & & \text{ in } \Omega, \\
& \varphi(t, x) = 0 & & \text{ in } [0, T] \times \partial\Omega,
\end{aligned}
\right.
\end{equation}
where $K: (0, T) \times \Omega \to \mathbb R$ is a given function.

\begin{definition}
\label{DefiWeakSolHJB}
Let $\nu > 0$, $K \in L^\infty((0, T) \times \Omega; \mathbb R)$, and $\psi \in L^2(\Omega)$. We say that $\varphi \in L^\infty((0, T); L^2(\Omega)) \cap L^2((0, T); H^1_0(\Omega))$ is a \emph{weak solution} of \eqref{HJBEqn} if, for every $\eta \in C^1([0, T] \times \Omega)$ such that $\eta\restrict{[0, T] \times \partial\Omega} = 0$ and $\eta\restrict{\{0\} \times \Omega} = 0$, one has
\begin{equation}\label{eq test}
\int_0^T \int_\Omega \varphi \partial_t \eta + \nu \int_0^T \int_{\Omega} \nabla\varphi \cdot \nabla\eta + \int_0^T \int_\Omega (K \abs*{\nabla\varphi} - 1) \eta = \int_\Omega \psi(x) \eta(T, x) \diff x.
\end{equation}
\end{definition}

As we did after Definition~\ref{DefiWeakSolFP}, we observe that, if \eqref{eq test} holds for every $\eta$ as before, and if we assume further that $\varphi \in C^0([t_1, t_2]; L^2(\Omega))$ for some $t_1, t_2 \in (0, T)$ with $t_1 < t_2$, then we also have
\begin{multline}
\label{eq testt_0t_1}
\int_{t_0}^{t_1} \int_\Omega \left(\varphi \partial_t \eta + \nu \nabla\varphi \cdot \nabla\eta + (K \abs*{\nabla\varphi} - 1) \eta \right) \\ = \int_\Omega \varphi(t_1,x) \eta(t_1, x) \diff x-\int_\Omega \varphi(t_0,x) \eta(t_0, x) \diff x,
\end{multline}
for every $\eta \in C^1_c((0, T] \times \Omega)$ and, by density, for every $\eta \in L^2((t_1, t_2); H_0^1(\Omega)) \cap C^0([t_1, t_2];\allowbreak L^2(\Omega))$ such that $\partial_t \eta\in L^2((t_1, t_2); H^{-1}(\Omega))$.

\begin{remark}
Note that \eqref{HJBEqn}, as a Hamilton--Jacobi--Bellman equation of an optimal control problem, is backward in time: the final condition $\varphi(T, x) = \psi(x)$ is given and one solves the equation in the time interval $[0, T]$. One can apply classical results on forward PDEs to~\eqref{HJBEqn} by using the standard time reversal $t \mapsto T - t$.
\end{remark}

The next proposition gathers the main results on solutions of \eqref{HJBEqn} that will be needed in the paper.

\begin{proposition}
\label{PropHJB}
Let $\nu > 0$, $K \in L^\infty((0, T) \times \Omega)$, and $\psi \in L^2(\Omega)$. Then \eqref{HJBEqn} admits a unique weak solution $\varphi$. In addition, we have $\varphi \in C^0([0, T]; L^2(\Omega))$, and the norms of $\varphi$ in $L^\infty((0, T);L^2(\Omega))$ and $L^2((0, T); H_0^1(\Omega))$ are bounded by quantities depending only on $d$, $\nu$, $T$, $\Omega$, an upper bound on $\norm*{K}_{L^\infty((0, T) \times \Omega)}$, and $\norm*{\psi}_{L^2(\Omega)}$.

Moreover, if $\psi\geq 0$ a.e.\ in $\O$, then the unique solution also satisfies $\varphi\geq 0$ a.e.\ in $(0,T)\times\O$. If $K \geq 0$, $\psi \in H_0^1(\Omega) \cap L^\infty(\Omega)$, and $\psi \geq 0$ a.e.\ in $\Omega$, then there exists a constant $C > 0$ depending on $\nu$, $\Omega$, and $\norm{\psi}_{L^\infty}$ such that $\varphi \leq C$ a.e.\ on $(0, T) \times \Omega$.

Finally, if $\psi \in H_0^1(\Omega)$, then $\varphi \in C^0([0, T]; H_0^1(\Omega)) \cap L^2((0, T);H^2(\Omega))$, $\partial_t \varphi \in L^2((0, T) \times \Omega)$, and the norms of $\varphi$ in these spaces are bounded by quantities depending only on $d$, $\nu$, $T$, $\Omega$, an upper bound on $\norm*{K}_{L^\infty((0, T) \times \Omega)}$, and $\norm*{\psi}_{H_0^1(\Omega)}$.
\end{proposition}

The results stated in Proposition~\ref{PropHJB} are classical and follow from more general results for nonlinear pseudo-monotone operators. Similarly to Proposition~\ref{PropFP}, we explain below how they can be retrieved from the relevant literature.

\begin{proof}
Existence of a weak solution $\varphi$ for $\psi \in L^2(\Omega)$ follows from \cite[Theorem~2.1]{Porzio1999Existence} and the corresponding bounds on the norms of $\varphi$ are a consequence of \cite[Lemma~4.1]{Porzio1999Existence}, whereas uniqueness follows from \cite[Theorem~2.4]{Feo2014Remark}. 

The positivity of $\varphi$ when $\psi \geq 0$ is classical for smooth solutions and can be obtained by an easy application of the maximum principle for parabolic equations. For solutions of HJB obtained as value functions of a stochastic control problem, the result is also straightforward, as the quantity which is minimized is positive. In our context of weak solutions, it can be deduced by applying, for instance, \cite[Theorem~1]{Aronson1967Local} to $-\varphi$, after changing time orientation and paying attention to the observation at the end of the proof (page~98) that the inequality is enough (indeed, the source term $1$ in the HJB equation has the good sign to preserve positivity).

The upper bound on $\varphi$ under the positivity assumption on $\psi$ and $K$ and the fact that $\psi \in H_0^1(\Omega) \cap L^\infty(\Omega)$ can be obtained by applying a parabolic comparison principle (see \cite[Theorem~9.1]{Lieberman1996Second} for the smooth case) to $\varphi$ and $\Phi + \norm{\psi}_{L^\infty}$, where $\Phi$ is the solution of the torsion equation $-\nu \Delta \Phi = 1$ in $\Omega$ with Dirichlet boundary conditions.

Finally, higher regularity of $\varphi$ when $\psi \in H_0^1(\Omega)$ can be obtained in a straightforward manner by noticing that $-\partial_t \varphi - \nu \Delta\varphi = 1 - K \abs{\nabla\varphi}$, i.e., $\varphi$ satisfies a linear backwards heat equation in $\Omega$ with source term $1 - K \abs{\nabla\varphi} \in L^2((0, T) \times \Omega)$. The conclusions then follow from classical improved regularity results for heat equations (such as \cite[Section~7.1, Theorem~5]{Evans1998Partial} and \cite[Chapter~III, \S~6, Equation~(6.10) and Theorem~6.1]{Ladyzhenskaja1968Linear}).
\end{proof}

We next state, for future reference, a standard parabolic comparison principle for \eqref{HJBEqn} (see, e.g., \cite[Corollary~2.2]{Feo2014Remark}).

\begin{proposition}\label{prop cp}
Let $\varphi_1,\varphi_2$ be two solutions of \eqref{HJBEqn} with $T<+\infty$, with final data such that $\varphi_1(T,\cdot)\geq \varphi_2(T,\cdot)$. Then
$$
\varphi_1\geq \varphi_2 \quad \text{ on } \ (0,T)\times\O.
$$
\end{proposition}

\section{The MFG system with a finite time horizon}
\label{SecMFGFinite}

We now consider the MFG system with a finite time horizon \eqref{MFGSystFiniteHorizon}. One of the difficulties in the analysis of \eqref{MFGSystFiniteHorizon} is that the velocity field in the continuity equation depends on $\frac{\nabla\varphi}{\abs{\nabla\varphi}}$, which is defined only when $\nabla\varphi \neq 0$. In order to handle this difficulty, we make use of the following definition of weak solution.

\begin{definition}
\label{DefiWeakSolMFG}
Let $\nu > 0$, $T \in (0, +\infty)$, $\kappa: \mathbb R \to (0, +\infty)$ be continuous and bounded, $\rho_0 \in L^1(\Omega)$, and $\psi \in L^2(\Omega)$. We say that $(\rho, \varphi) \in L^1((0, T)\times\Omega)\times L^2((0, T); H_0^1(\Omega))$ is a \emph{weak solution} of \eqref{MFGSystFiniteHorizon} with initial condition $\rho_0$ and final condition $\psi$ if there exists $V \in L^\infty((0, T) \times \Omega; \mathbb R^d)$ such that $\abs{V(t, x)} \leq \kappa(\rho(t, x))$ and $V(t, x) \cdot \nabla\varphi(t, x) = - \kappa(\rho(t, x)) \abs*{\nabla\varphi(t, x)}$ a.e.\ on $(0, T) \times \Omega$ and 
such that $\rho$ is a solution of the Fokker--Planck equation \eqref{FPEqn} with initial datum $\rho_0$ and vector field $V$ on $[0,T]\times\Omega$ in the sense of Definition~\ref{DefiWeakSolFP}, and $\varphi$ is a solution of the Hamilton--Jacobi--Bellman equation \eqref{HJBEqn} with final datum $\psi$ and  $K=\kappa(\rho)$ in the sense of Definition~\ref{DefiWeakSolHJB} on the same domain.
\end{definition}

\begin{remark}
If $(\rho, \varphi)$ is a weak solution of \eqref{MFGSystFiniteHorizon} and $V$ is any function satisfying the properties stated in Definition~\ref{DefiWeakSolMFG}, then we have  $V(t, x) = -\kappa(\rho(t, x)) \frac{\nabla\varphi(t, x)}{\abs{\nabla\varphi(t, x)}}$ wherever $\nabla\varphi(t, x) \neq 0$. The introduction of the function $V$ in Definition~\ref{DefiWeakSolMFG} has the advantages of providing a meaning to the first equation of \eqref{MFGSystFiniteHorizon} and handling its velocity field even when $\nabla\varphi(t, x) = 0$, which might a priori happen in a set of positive measure.
\end{remark}

The main result of this section is the following.

\begin{theorem}
\label{MainTheoFixedT}
Let $\nu > 0$, $T \in (0, +\infty)$, $\kappa: \mathbb R \to (0, +\infty)$ be continuous and bounded, $\rho_0 \in L^1(\Omega)$, and $\psi \in H_0^1(\Omega)$. Then there exists a weak solution $(\rho, \varphi)$ of \eqref{MFGSystFiniteHorizon} with initial condition $\rho_0$ and final condition $\psi$.
\end{theorem}

The proof of Theorem~\ref{MainTheoFixedT} relies on a fixed-point argument on the velocity field $V$ of the Fokker--Planck equation in \eqref{MFGSystFiniteHorizon}. Before turning to the proof, we need some continuity results on solutions of \eqref{FPEqn} with respect to the velocity field $V$ and on solutions of \eqref{HJBEqn} with respect to the function $K$, which we state and prove now.

\begin{proposition}
\label{PropContinuityFP}
Let $\nu > 0$ and $\rho_0 \in L^1(\Omega)$. Given $V \in L^\infty((0, T) \times \Omega; \mathbb R^d)$, let $(V_n)_{n \in \mathbb N}$ be a sequence in $L^\infty((0, T) \times \Omega; \mathbb R^d)$ such that $V_n \xrightharpoonup{\ast} V$ as $n \to \infty$. For $n \in \mathbb N$, let $\rho_n$ (resp.\ $\rho$) be the unique weak solution of \eqref{FPEqn} in $L^1((0, T)\times\Omega)$ with velocity field $V_n$ (resp.\ $V$). Then $\rho_n \to \rho$ in $L^1((0, T) \times \Omega)$ as $n \to \infty$.
\end{proposition}

\begin{proof}
Since $(V_n)_{n \in \mathbb N}$ converges weakly-$\ast$ to $V$ in $L^\infty$, there exists a constant $M > 0$ such that $\norm*{V_n}_{L^\infty((0, T) \times \Omega)} \leq M$ for every $n \in \mathbb N$ and thus, by Proposition~\ref{PropFP}, there exists $C > 0$ depending only on $d$, $\nu$, $M$, and $\norm*{\rho_0}_{L^1(\Omega)}$ such that, for every $n \in \mathbb N$,
\begin{equation}
\label{BoundL2H01}
\norm*{\rho_n}_{L^\infty((0, T);L^1(\Omega))} + \norm*{\rho_n}_{L^q((0, T); W^{1,q}(\Omega))}+\norm*{\partial_t \rho_n}_{L^q((0, T);W^{-1,q}(\Omega))}  \leq C.
\end{equation}
It follows from \eqref{BoundL2H01} and Aubin--Lions Lemma (see, e.g., \cite[Corollary 4]{Simon1987Compact}) that $(\rho_n)_{n \in \mathbb N}$ is relatively compact in $L^1((0, T) \times \Omega)$. Let $\rho^\ast \in L^1((0, T) \times \Omega)$ be a limit point of $(\rho_n)_{n \in \mathbb N}$ and $(\rho_{n_k})_{k \in \mathbb N}$ a subsequence of $(\rho_n)_{n \in \mathbb N}$ converging to $\rho^\ast$ in $L^1((0, T) \times \Omega)$.

The weak convergence of $V_n$ in $L^\infty$ together with the strong convergence of $\rho_n$ in $L^1$ allow to pass to the limit the drift term $\nabla\cdot(\rho_n V_n)$ in the equation and we then easily obtain that $\rho^\ast$ is a weak solution of \eqref{FPEqn}. By the uniqueness of such solution from Proposition~\ref{PropFP}, one concludes $\rho^\ast = \rho$. In particular, $\rho$ is the unique limit point of the relatively compact sequence $(\rho_n)_{n \in \mathbb N}$ in $L^1((0, T) \times \Omega)$, which yields the result.
\end{proof}

\begin{proposition}
\label{PropContinuityHJB}
Let $\nu > 0$ and $\psi \in H_0^1(\Omega)$. Given $K \in L^\infty((0, T) \times \Omega)$, let $(K_n)_{n \in \mathbb N}$ be a sequence in $L^\infty((0, T) \times \Omega)$ such that $K_n \xrightharpoonup{\ast} K$ as $n \to \infty$. For $n \in \mathbb N$, let $\varphi_n$ (resp.\ $\varphi$) be the unique weak solution of \eqref{HJBEqn} in $L^\infty((0, T);L^2(\Omega)) \cap L^2((0, T); H_0^1(\Omega))$ with $K_n$ (resp.\ $K$). Then $\varphi_n \to \varphi$ in $L^2((0, T); H_0^1(\Omega))$ as $n \to \infty$.
\end{proposition}

\begin{proof}
Again, there exists a constant $M > 0$ such that $\norm*{K_n}_{L^\infty((0, T) \times \Omega)} \leq M$ for every $n \in \mathbb N$ and thus, by Proposition~\ref{PropHJB}, there exists $C > 0$ depending only on $d$, $\nu$, $T$, $\Omega$, $M$, and $\norm*{\psi}_{H_0^1(\Omega)}$ such that
\begin{equation}
\label{EstimateVarphiN}
\norm*{\varphi_n}_{L^\infty((0, T); H_0^1(\Omega))} + \norm*{\varphi_n}_{L^2((0, T);H^2(\Omega))} + \norm*{\partial_t \varphi_n}_{L^2((0, T) \times \Omega)} \leq C.
\end{equation}
 Hence, by Aubin--Lions Lemma (see, e.g., \cite[Corollary~4]{Simon1987Compact}), $(\varphi_n)_{n \in \mathbb N}$ is relatively compact in $L^2((0, T); H_0^1(\Omega))$. Let $\varphi^\ast$ be a limit point of $(\varphi_n)_{n \in \mathbb N}$ and $(\varphi_{n_k})_{k \in \mathbb N}$ be a subsequence of $(\varphi_n)_{n \in \mathbb N}$ converging to $\varphi^\ast$ in $L^2((0, T); H_0^1(\Omega))$. By \eqref{EstimateVarphiN}, we also have $\varphi^\ast \in L^\infty((0, T); \allowbreak H_0^1(\Omega)))$.

Now, because of the non-linearity in the equation, we prefer to provide details on how to pass it to the limit.
For every $k$ and every $\eta \in H^1((0, T) \times \Omega)$ such that $\eta\restrict{[0, T] \times \partial\Omega} = 0$ and $\eta\restrict{\{0\} \times \Omega} = 0$, one has
\[
\int_0^T \int_\Omega \varphi_{n_k} \partial_t \eta + \nu \int_0^T \int_{\Omega} \nabla\varphi_{n_k} \cdot \nabla\eta + \int_0^T \int_\Omega (K_{n_k} \abs*{\nabla\varphi_{n_k}} - 1) \eta = \int_\Omega \psi(x) \eta(T, x) \diff x.
\]
Since $K_{n_k} \xrightharpoonup{\ast} K$ in $L^\infty((0, T) \times \Omega)$ and $\varphi_{n_k} \to \varphi^\ast$ in $L^2((0, T); H_0^1(\Omega))$, one obtains, letting $k \to \infty$, that
\[
\int_0^T \int_\Omega \varphi^\ast \partial_t \eta + \nu \int_0^T \int_{\Omega} \nabla\varphi^\ast \cdot \nabla\eta + \int_0^T \int_\Omega (K \abs*{\nabla\varphi^\ast} - 1) \eta = \int_\Omega \psi(x) \eta(T, x) \diff x.
\]
Hence $\varphi^\ast \in L^\infty((0, T);L^2(\Omega)) \cap L^2((0, T); H_0^1(\Omega))$ is a weak solution of \eqref{HJBEqn} and, by the uniqueness of solutions of \eqref{HJBEqn} from Proposition~\ref{PropHJB}, one deduces that $\varphi^\ast = \varphi$. Thus $\varphi$ is the unique limit point in $L^2((0, T); H_0^1(\Omega))$ of the relatively compact sequence $(\varphi_n)_{n \in \mathbb N}$, yielding the conclusion.
\end{proof}

We now recall the statement of Kakutani's fixed point theorem (see, e.g., \cite[\S~7, Theorem~8.6]{Granas2003Fixed}), which is used in the proof of Theorem~\ref{MainTheoFixedT}.

\begin{theorem}[Kakutani's fixed point theorem]
Let $\mathcal B$ be a compact convex subset of a locally convex topological vector space $E$, and let $\mathcal V$ be a set-valued map in $\mathcal B$, i.e., $\mathcal V$ associates, with each $b \in \mathcal B$, a set $\mathcal V(b) \subset \mathcal B$. Assume that $\mathcal V$ is upper semi-continuous and that, for every $b \in \mathcal B$, $\mathcal V(b)$ is non-empty, compact, and convex. Then $\mathcal V$ admits a fixed point in $\mathcal B$, i.e., there exists $b \in \mathcal B$ such that $b \in S(b)$.
\end{theorem}

We are finally in position to provide the proof of Theorem~\ref{MainTheoFixedT}.

\begin{proof}[Proof of Theorem~\ref{MainTheoFixedT}]
Let $\kappa_0$ be an upper bound on $\kappa$. We endow the space $L^\infty((0, T) \times \Omega; \mathbb R^d)$ with its weak-$\ast$ topology and consider the ball of radius $\kappa_0$ given by
\[\mathcal B = \left\{V \in L^\infty((0, T) \times \Omega; \mathbb R^d) \suchthat \norm{V}_{L^\infty((0, T) \times \Omega; \mathbb R^d)} \leq \kappa_0\right\}.\]
Note that $\mathcal B$ is clearly convex and, by the Banach--Alaoglu theorem, $\mathcal B$ is a compact subset of $L^\infty((0, T) \times \Omega; \mathbb R^d)$.

Let $\mathcal S_{\mathrm{FP}}: L^\infty((0, T) \times \Omega; \mathbb R^d) \to L^1((0, T) \times \Omega)$ be the function that associates, with each $V \in L^\infty((0, T) \times \Omega; \mathbb R^d)$, the unique weak solution $\rho = \mathcal S_{\mathrm{FP}}(V) \in L^1((0, T) \times \Omega)$ of \eqref{FPEqn} with initial condition $\rho_0$. Note that, by Proposition~\ref{PropContinuityFP}, $\mathcal S_{\mathrm{FP}}$ is continuous with respect to the weak-$\ast$ topology of $L^\infty((0, T) \times \Omega; \mathbb R^d)$ and the strong topology of $L^1((0, T) \times \Omega)$. Similarly, we define $\mathcal S_{\mathrm{HJB}}: L^\infty((0, T) \times \Omega) \to L^2((0, T); H_0^1(\Omega))$ as the function that associates, with each $K \in L^\infty((0, T) \times \Omega)$, the unique weak solution $\varphi = \mathcal S_{\mathrm{HJB}}(K) \in L^2((0, T); H_0^1(\Omega))$ of \eqref{HJBEqn} with terminal condition $\psi$. Proposition~\ref{PropContinuityHJB} ensures that $\mathcal S_{\mathrm{HJB}}$ is continuous with respect to the weak-$\ast$ topology of $L^\infty((0, T) \times \Omega)$ and the strong topology of $L^2((0, T); H_0^1(\Omega))$.

We define the set-valued map $\mathcal V$ that, with each $V \in \mathcal B$, associates the set $\mathcal V(V) \subset \mathcal B$ given by
\begin{align*}
\mathcal V(V) = \Big\{\widetilde V \in \mathcal B \mathrel{}\Big\vert\mathrel{} & \abs[\big]{\widetilde V(t, x)} \leq \kappa(\rho(t, x)) \text{ for a.e.\ } (t, x) \in (0, T) \times \Omega, \\
 & \widetilde V(t, x) \cdot \nabla\varphi(t, x) = - \kappa(\rho(t, x)) \abs*{\nabla\varphi(t, x)} \text{ for a.e.\ } (t, x) \in (0, T) \times \Omega, \\
 & \text{where } \rho = \mathcal S_{\mathrm{FP}}(V) \text{ and } \varphi = \mathcal S_{\mathrm{HJB}}(\kappa \circ \rho)\Big\}.
\end{align*}
In order to prove the existence of a weak solution $(\rho, \varphi)$ of \eqref{MFGSystFiniteHorizon}, we first prove the existence of a fixed point of the set-valued map $\mathcal V$, i.e., of a $V \in L^\infty((0, T) \times \Omega; \mathbb R^d)$ such that $V \in \mathcal V(V)$. This is done by applying Kakutani's fixed point theorem to the set-valued map $\mathcal V$. To do so, we first need to verify some properties of $\mathcal V$ and its graph $\mathcal G$ defined by
\[
\mathcal G = \left\{(V, \widetilde V) \in \mathcal B \times \mathcal B \suchthat \widetilde V \in \mathcal V(V)\right\}.
\]

\begin{claim}
\label{ClaimNonemptyConvex}
For every $V \in L^\infty((0, T) \times \Omega; \mathbb R^d)$, the set $\mathcal V(V)$ is non-empty and convex.
\end{claim}

\begin{proof}
It is immediate to verify that $\mathcal V(V)$ is convex. To prove that it is non-empty, let $V \in L^\infty((0, T) \times \Omega; \mathbb R^d)$, $\rho = \mathcal S_{\mathrm{FP}}(V)$, and $\varphi = \mathcal S_{\mathrm{HJB}}(\kappa \circ \rho)$. Then, the function $\widetilde V \in L^\infty((0, T) \times \Omega; \mathbb R^d)$ defined for a.e.\ $(t, x) \in (0, T) \times \Omega$ by
\[
\widetilde V(t, x) = \begin{dcases*}
-\kappa(\rho(t, x)) \frac{\nabla\varphi(t, x)}{\abs*{\nabla\varphi(t, x)}} & if $\nabla\varphi(t, x) \neq 0$, \\
0 & otherwise,
\end{dcases*}
\]
clearly satisfies $\widetilde V \in \mathcal V(V)$.
\end{proof}

\begin{claim}
\label{ClaimClosedGraph}
The graph $\mathcal G$ is a closed subset of $\mathcal B \times \mathcal B$.
\end{claim}

\begin{proof}
Let $(V_n, \widetilde V_n)_{n \in \mathbb N}$ be a sequence in $\mathcal G$ converging weakly-$\ast$ in $\mathcal B \times \mathcal B$ to a point $(V, \widetilde V)$. We want to prove $(V, \widetilde V) \in \mathcal G$, i.e., $\widetilde V \in \mathcal V(V)$.

Define, for $n \in \mathbb N$, the functions $\rho_n \in L^1((0, T) \times \Omega)$ and $\varphi_n \in L^2((0, T); H_0^1(\Omega))$ by $\rho_n = \mathcal S_{\mathrm{FP}}(V_n)$ and $\varphi_n = \mathcal S_{\mathrm{HJB}}(\kappa \circ \rho_n)$ and, similarly, let $\rho = \mathcal S_{\mathrm{FP}}(V)$ and $\varphi = \mathcal S_{\mathrm{HJB}}(\kappa \circ \rho)$. Since $\mathcal S_{\mathrm{FP}}: L^\infty((0, T) \times \Omega; \mathbb R^d) \to L^1((0, T) \times \Omega)$ is continuous with respect to the weak-$\ast$ topology of $L^\infty((0, T) \times \Omega; \mathbb R^d)$ and the strong topology of $L^1((0, T) \times \Omega)$, one deduces $\rho_n \to \rho$ in $L^1((0, T) \times \Omega)$ as $n \to \infty$. Hence, up to extracting subsequences (which we still denote using the same notation for simplicity), one has $\rho_n \to \rho$ a.e.\ in $(0, T) \times \Omega$. Since $\kappa$ is continuous, we deduce $\kappa \circ \rho_n \to \kappa \circ \rho$ a.e.\ in $(0, T) \times \Omega$, and it follows $\kappa \circ \rho_{n} \xrightharpoonup{\ast} \kappa \circ \rho$ in $L^\infty((0, T) \times \Omega)$. The continuity of $\mathcal S_{\mathrm{HJB}}: L^\infty((0, T) \times \Omega; \mathbb R^d) \to L^2((0, T); H_0^1(\Omega))$ with respect to the weak-$\ast$ topology of $L^\infty((0, T) \times \Omega; \mathbb R^d)$ and the strong topology of $L^2((0, T); H_0^1(\Omega))$ implies $\varphi_n \to \varphi$ in $L^2((0, T); H_0^1(\Omega))$ as $n \to \infty$.

From the weak convergence of $\widetilde V_n$ to $\widetilde V$, the convexity of the function $\abs{\cdot}$, and the (strong) convergence of $\kappa(\rho_n)$ to $\kappa(\rho)$, the inequality $\abs[\big]{\widetilde V_n} \leq \kappa(\rho_n)$ gives at the limit 
\begin{equation}
\label{TildeVInVV-FirstStep}
\abs[\big]{\widetilde V(t, x)} \leq \kappa(\rho(t, x)) \qquad \text{ for a.e.\ } (t, x) \in (0, T) \times \Omega.
\end{equation}

Since $\widetilde V_n \in \mathcal V(V_n)$ for every $n \in \mathbb N$, we have $\widetilde V_n(t, x) \cdot \nabla\varphi_n(t, x) = - \kappa(\rho_n(t, x)) \abs*{\nabla\varphi_n(t, x)}$ for a.e.\ $(t, x) \in (0, T) \times \Omega$. Then, for every $v \in L^2((0, T) \times \Omega)$, one has
\[
\int_0^T \int_\Omega \widetilde V_n(t, x) \cdot \nabla\varphi_n(t, x) v(t, x) \diff x \diff t = - \int_0^T \int_\Omega \kappa(\rho_n(t, x)) \abs*{\nabla\varphi_n(t, x)} v(t, x) \diff x \diff t.
\]
Recalling that, as $n \to \infty$, one has $\widetilde V_n \xrightharpoonup{\ast} \widetilde V$ in $L^\infty((0, T) \times \Omega)$, $\nabla\varphi_n \to \nabla\varphi$ in $L^2((0, T) \times \Omega)$, and $\kappa \circ \rho_n \xrightharpoonup{\ast} \kappa \circ \rho$ in $L^\infty((0, T) \times \Omega)$, we obtain, letting $n \to \infty$, that
\[
\int_0^T \int_\Omega \widetilde V(t, x) \cdot \nabla\varphi(t, x) v(t, x) \diff x \diff t = - \int_0^T \int_\Omega \kappa(\rho(t, x)) \abs*{\nabla\varphi(t, x)} v(t, x) \diff x \diff t
\]
for every $v \in L^2((0, T) \times \Omega)$, which implies that
\begin{equation}
\label{TildeVInVV-SecondStep}
\widetilde V(t, x) \cdot \nabla\varphi(t, x) = - \kappa(\rho(t, x)) \abs*{\nabla\varphi(t, x)} \qquad \text{ for a.e.\ } (t, x) \in (0, T) \times \Omega.
\end{equation}
Combining \eqref{TildeVInVV-FirstStep} and \eqref{TildeVInVV-SecondStep}, we conclude that $\widetilde V \in \mathcal V(V)$, as required.
\end{proof}

\begin{claim}
\label{ClaimCompact}
For every $V \in L^\infty((0, T) \times \Omega; \mathbb R^d)$, the set $\mathcal V(V)$ is compact.
\end{claim}

\begin{proof}
This is a consequence of the fact that $\mathcal G$ is a closed subset of the compact set $\mathcal B \times \mathcal B$.
\end{proof}

Thanks to Claims~\ref{ClaimClosedGraph} and \ref{ClaimCompact}, it follows from \cite[Proposition~1.4.8]{Aubin2009Set} that the set-valued map $\mathcal V$ is upper semi-continuous. Using this fact and Claims~\ref{ClaimNonemptyConvex} and \ref{ClaimCompact}, it follows from Kakutani's fixed point theorem that $\mathcal V$ admits a fixed point $V \in \mathcal B$. Let $\rho = \mathcal S_{\mathrm{FP}}(V)$ and $\varphi = \mathcal S_{\mathrm{HJB}}(\kappa \circ \rho)$. Using the facts that $\rho$ and $\varphi$ are solutions of \eqref{FPEqn} and \eqref{HJBEqn}, respectively, and that $V \in \mathcal V(V)$, it is immediate to verify, using Definitions~\ref{DefiWeakSolFP}, \ref{DefiWeakSolHJB}, and \ref{DefiWeakSolMFG}, that $(\rho, \varphi)$ is a weak solution of \eqref{MFGSystFiniteHorizon} with initial condition $\rho_0$ and final condition $\psi$, as required.
\end{proof}

\section{The MFG system with an infinite time horizon}
\label{SecMFGInfinite}

Now that we have established in Section~\ref{SecMFGFinite} the existence of solutions to the Mean Field Game system \eqref{MFGSystFiniteHorizon} in a finite time horizon $T$, we consider in this section the Mean Field Game system \eqref{MFGSystInfiniteHorizon} with an infinite time horizon. Let us first provide the definition of a weak solution in this setting.

\begin{definition}
Let $\nu > 0$, $\kappa: \mathbb R \to (0, +\infty)$ be continuous and bounded, and $\rho_0 \in L^1(\Omega)$.
We say that $(\rho, \varphi) \in L^\infty_{\mathrm{loc}}(\mathbb R_+; L^1(\Omega)) \times L^2_{\mathrm{loc}}(\mathbb R_+; H_0^1(\Omega))$ is a \emph{weak solution} of \eqref{MFGSystInfiniteHorizon} with initial condition $\rho_0$ if $\varphi \in L^\infty(\mathbb R_+ \times \Omega)$ and if there exists $V \in L^\infty(\mathbb R_+ \times \Omega; \mathbb R^d)$ such that $\abs{V(t, x)} \leq \kappa(\rho(t, x))$ and $V(t, x) \cdot \nabla\varphi(t, x) = - \kappa(\rho(t, x)) \abs*{\nabla\varphi(t, x)}$ a.e.\ on $\mathbb R_+ \times \Omega$ and such that, for every $T>0$, $\rho$ is a solution of the Fokker--Planck equation \eqref{FPEqn} with initial datum $\rho_0$ and vector field $V$ on $[0,T]\times\Omega$ in the sense of Definition~\ref{DefiWeakSolFP} and $\varphi$ is a solution of the Hamilton--Jacobi--Bellman equation \eqref{HJBEqn} with $K=\kappa(\rho)$ in the sense of Definition~\ref{DefiWeakSolHJB} on the same domain\footnote{Note that Definition~\ref{DefiWeakSolHJB} requires to fix a final value, and we did not define the notion of solution independently of the final value $\psi$. This could be formalized as ``there exists $\psi\in L^2(\Omega)$ such that $\varphi$ is a solution of \eqref{HJBEqn}''. Yet, since the function $\varphi$ will be finally continuous as a function valued into $L^2(\Omega)$, the final datum on $[0,T]$ will be necessarily given by its own value $\varphi(T, \cdot)$.}.
\end{definition}

Notice that, with respect to Definition~\ref{DefiWeakSolMFG}, we make the additional requirement that $\varphi \in L^\infty(\mathbb R_+ \times \Omega)$. This is done mainly for three reasons. Firstly, boundedness of the solution of a Hamilton--Jacobi--Bellman equation is a condition usually required in order to ensure that this solution is the value function of an optimal control problem (see, e.g., \cite[Theorem~8.1.10]{Cannarsa2004Semiconcave} and \cite[Chapter~II, Corollary~9.1]{Fleming2006Controlled}). Secondly, the strategy we use in this section to prove existence of a solution of \eqref{MFGSystInfiniteHorizon}, based on a limit argument from solutions of \eqref{MFGSystFiniteHorizon} in finite time horizon $T$ as $T \to +\infty$, allows us to ensure that the function $\varphi: \mathbb R_+ \times \Omega \to \mathbb R$ we construct is indeed bounded. Finally, boundedness of $\varphi$ is an important property in order to establish the results on the the asymptotic behavior of solutions to \eqref{MFGSystInfiniteHorizon} provided in Theorem~\ref{th long time} and Propositions~\ref{PropExponentialConvergence} and \ref{prop:VarphiConvergesInH1}.

\subsection{Existence of solutions and their asymptotic behavior}

From now on, we let $\Psi$ denote the solution of the (stationary) Hamilton--Jacobi--Bellman equation
\begin{equation}
\label{eq:HJB-Psi}
-\nu\Delta \Psi + \kappa(0)\vert \nabla  \Psi \vert = 1, \quad x\in \O,
\end{equation}
with Dirichlet boundary conditions $\Psi=0$ on $\partial\Omega$. Existence of such a solution $\Psi$ follows from standard results on elliptic equations, and $\Psi$ is continuous in the closure of $\O$ and $C^2$ and positive in $\O$ (see, e.g., \cite[Theorem~15.12]{Gilbarg2001Elliptic}, \cite[Chapter~4, Section~8]{Ladyzhenskaya1968Linear}, \cite{Lions1980Resolution}; these results require additional regularity properties on $\partial\Omega$ but they can be easily adapted to a $C^2$ boundary using the techniques from \cite[Section~15.6]{Gilbarg2001Elliptic} or \cite[Chapter~4, pp.~309--310]{Ladyzhenskaya1968Linear}). Uniqueness of $\Psi$ follows also from classical arguments for elliptic equations based on the maximum principle: the difference $\widetilde\Psi$ of two solutions of \eqref{eq:HJB-Psi} is zero on $\partial\Omega$ and satisfies $-\nu \Delta\widetilde\Psi - \kappa(0) \abs{\nabla\widetilde\Psi} \leq 0$ in $\Omega$, and hence the maximum principle from \cite[Theorem~10.9]{Gilbarg2001Elliptic} allows one to conclude that $\widetilde\Psi = 0$ in $\Omega$.

The main result of this section is the following.

\begin{theorem}\label{th long time}
Let $\rho_0 \in L^1(\Omega)$. Then, there exists at least one solution $(\rho,\varphi)$ to the Mean Field Game system with infinite time horizon \eqref{MFGSystInfiniteHorizon}.

In addition, any such solution satisfies
$$
\rho_t \underset{t \to +\infty}{\longrightarrow} 0, \quad \varphi_t \underset{t\to +\infty}{\longrightarrow} \Psi,
$$
and the above convergences hold uniformly.
\end{theorem}

The sequel of this section is devoted to the proof of Theorem~\ref{th long time}. Let us start by giving an idea of the proof. First, we will construct solutions to the problem with infinite time horizon as limits of solutions of the problem with finite time horizon $T$ by letting $T$ go to $+\infty$. Then, to prove the long-time uniform convergence of the solutions, we shall make a crucial use of some regularity results for parabolic equations. More precisely, we will use local maximum principles for Fokker--Planck and for (forward) Hamilton--Jacobi--Bellman equations; roughly speaking, these results state that the $L^\infty(\O)$ norm of solutions of such equations at some time $t_2$ is controlled by some $L^p$ norms of the same solution at some previous time $t_1<t_2$. The results we use are proved in Appendix~\ref{AnnexeA}, see Proposition~\ref{prop L infty appendix} and Corollaries~\ref{reg FP} and \ref{reg HJB}.

We start with a lemma that gathers some useful estimates. These estimates have already been discussed in Section~\ref{SecPreliminary}, but we need now to track possible dependencies of the constant on the time horizon $T$.

\begin{lemma}\label{lem est phi}
Let $(\rho,\varphi)$ be solution of the finite horizon MFG system \eqref{MFGSystFiniteHorizon} on $[0,T]\times\O$ in the sense of Definition~\ref{DefiWeakSolMFG}, with final datum $\psi \in H_0^1(\Omega) \cap L^\infty(\Omega)$ with $\psi \geq 0$. Then, there are $C_1,C_2>0$, depending on $\norm{\psi}_{L^\infty}+\norm{\psi}_{H^1_0}$, $\sup \kappa$, $\nu$, $\Psi$ and $\O$ such that
\begin{equation}
\label{EqLInftyH1BoundVarphi}
\norm{\nabla \varphi(t,\cdot) }_{L^2}\leq C_1, \quad \text{for all }\ t \in [0,T],
\end{equation}
and
\begin{equation}
\label{EqL2H2BoundVarphi}
\norm{ \varphi }_{L^2((T_1,T_2);H^2)} \leq C_2(1+\vert T_2-T_1\vert).
\end{equation}
\end{lemma}
\begin{proof}

\step{1}{A preliminary estimate}

Let us start with giving an estimate on the gradient of $\varphi$. First, multiplying by $\varphi$ the equation satisfied by $\varphi$ and integrating on $\O$ for a fixed $t\in (0,T)$, we find
$$
-\frac{\diff}{\diff t}\left(\frac{1}{2}\int_\O \varphi^2 \right)= -\nu \int_\O \vert \nabla \varphi \vert^2-\int_\O \kappa(\rho) \vert \nabla \varphi \vert \varphi +\int_\O\varphi.
$$
Therefore, since $\varphi$ is bounded, we have
\begin{equation}\label{est grad phi}
\int_{T_1}^{T_2} \int_\O\vert \nabla \varphi\vert^2 \leq C(1+\vert T_2-T_1\vert),
\end{equation}
for some $C>0$ depending on $\sup \varphi$, $\lvert \O\rvert$, $\sup \kappa$, $\nu$
and for every $T_1, T_2 \in [0, T]$ with $0 \leq T_1 \leq T_2 \leq T$.
Note that, from Proposition~\ref{PropHJB}, $\sup \varphi$ is bounded in terms of $\nu$, $\Omega$, and $\norm{\psi}_{L^\infty}$.

\step{2}{Bound on $\norm{ \nabla \varphi(t,\cdot)}_{L^2}$}

We define, for $t\in[0,T]$,
$$
u(t) := \frac{1}{2}\int_{\O}\vert \nabla \varphi(t,x)\vert^2 \diff x.
$$
We differentiate $u$ to obtain
\begin{equation}\label{est pour H2}
u^\prime(t) = \nu \int_\O (\Delta \varphi)^2 - \int_\O \kappa(\rho) \lvert \nabla \varphi\rvert \Delta \varphi + \int_\O \Delta \varphi .
\end{equation}
Using Young's inequality, we find that there are $K_1,K_2>0$ depending only on $\lvert \O\rvert$, $\sup \kappa$, $\nu$ such that
$$
u^\prime(t) +K_1 u(t) +  K_2 \geq 0.
$$
This implies, for any $0\leq t<s\leq T$,
\begin{equation}\label{a int}
u(s)+ \frac{K_2}{K_1}\left(1 - e^{- K_1(s - t)}\right) \geq u(t)e^{-K_ 1(s -t)}.
\end{equation}
We integrate \eqref{a int} for $s \in (t, t+1)$ to get
\[
\frac{1}{2}\int_{t}^{t+1}\int_{\O}\vert \nabla \varphi \vert^2(s,x)\diff x \diff s + \frac{K_2}{K_1}\int_0^1 (1-e^{-K_1r})\diff r \geq \frac{1}{2} \left(\int_0^1e^{-K_1r}\diff r\right) \int_{\O}\vert \nabla \varphi \vert^2(t,x)\diff x .
\]
Using \eqref{est grad phi} yields the $L^\infty(H^1)$ bound \eqref{EqLInftyH1BoundVarphi} for $t\in [0,T-1)$. To get the $L^\infty(H^1)$ bound \eqref{EqLInftyH1BoundVarphi} for $t\in [T-1,T]$, we use \eqref{a int} with $s = T$. The result follows, with a constant also depending on $u(T) = \frac{1}{2}\int_\O \vert \nabla \psi \vert^2 <+\infty$.

\step{3}{Bound in $L^2((T_1,T_2);H^2)$}

Let us integrate \eqref{est pour H2} on $(T_1,T_2)$. We find
$$
\nu \int_{T_1}^{T_2}\int_\O (\Delta \varphi)^2 = u(T_2)-u(T_1) + \int_{T_1}^{T_2}\int_\O \kappa(\rho) \vert \nabla \varphi \vert\Delta \varphi - \int_{T_1}^{T_2}\int_\O \Delta \varphi.
$$
Using Young's inequality on $\int_{T_1}^{T_2}\int_\O \kappa(\rho) \vert \nabla \varphi \vert\Delta \varphi$ and $\int_{T_1}^{T_2}\int_\O \Delta \varphi$ and the estimate \eqref{est grad phi}, we get the desired bound \eqref{EqL2H2BoundVarphi} on $L^2((T_1,T_2);H^2)$.
\end{proof}

The next lemma shows that the time derivative of $\int_\O \rho \varphi$ is equal to $-\int_\O \rho$. Differentiating the average value of the value function is a classical computation in Mean Field Game theory. Since here the value function is an exit time, it is expected that it should decrease with rate $1$, and one can guess the result from the fact that the total mass of the agents in this model is not fixed but decreases in time and is equal to $\int_\O \rho$.

\begin{lemma}\label{lem rho phi}
Let $(\rho,\varphi)$ be a solution of the finite-horizon MFG \eqref{MFGSystFiniteHorizon} on $[0,T]\times\O$ in the sense of Definition~\ref{DefiWeakSolMFG}. Then, for a.e.\ $t$, we have
$$
\frac{\diff}{\diff t} \left(\int_\O \rho(t,x) \varphi(t,x)\diff x\right) = -\int_\O \rho(t,x)\diff x.
$$
\end{lemma}
\begin{proof}
Let us fix two instants of times $t_1<t_2\leq T$, with $t_1>0$. On the interval $(t_1,t_2)$ we can use $\varphi$ as a test function in \eqref{FPWeakH1_2} and $\rho$ in \eqref{eq testt_0t_1} since both $\varphi$ and $\rho$ are continuous as curves valued in $L^2$, belong to $L^2((t_1,t_2); H^1_0(\O))$, and their time-derivatives belong to $L^2((t_1,t_2); H^{-1}(\O))$. We subtract the two equalities that we obtain, which provides
\begin{multline*}
    \int_{t_1}^{t_2} \int_\Omega \rho \partial_t \varphi \diff x \diff t - \int_{t_1}^{t_2}  \int_{\Omega} \left(\nu \nabla\rho - \rho V\right) \cdot \nabla\varphi \diff x \diff t \\+
\int_{t_1}^{t_2} \int_\Omega \left(\varphi \partial_t \rho + \nu \nabla\varphi \cdot \nabla\rho + (\kappa(\rho) \abs*{\nabla\varphi} - 1) \rho \right)\\= 2\int_\Omega \varphi(t_2,x) \rho(t_2, x) \diff x-2\int_\Omega \varphi(t_1,x) \rho(t_1, x) \diff x.
\end{multline*}
After canceling the terms with $\nabla\rho\cdot\nabla\varphi$ and using $V\cdot\nabla\varphi+\kappa(\rho)\abs{\nabla\varphi}=0$ we are left with
\begin{multline*}
 \int_{t_1}^{t_2} \int_\Omega (\rho \partial_t \varphi +\varphi\partial_t\rho)\diff x \diff t 
-\int_{t_1}^{t_2} \int_\Omega \rho\diff x \diff t \\ = 2\int_\Omega \varphi(t_2,x) \rho(t_2, x) \diff x-2\int_\Omega \varphi(t_1,x) \rho(t_1, x) \diff x.
\end{multline*}
It is then easy to see, by approximation via smooth functions, that for every pair $(\rho,\varphi)$ such that $\rho,\varphi\in L^2((t_1, t_2);H^1(\O))$ and $\partial_t\rho,\partial_t\varphi\in L^2((t_1, t_2);H^{-1}(\O))$, we have 
$$\int_{t_1}^{t_2} \int_\Omega (\rho \partial_t \varphi +\varphi\partial_t\rho)\diff x \diff t =\int_\Omega \varphi(t_2,x) \rho(t_2, x) \diff x-\int_\Omega \varphi(t_1,x) \rho(t_1, x) \diff x.$$
We are then left with
$$\int_\Omega \varphi(t_2,x) \rho(t_2, x) \diff x-\int_\Omega \varphi(t_1,x) \rho(t_1, x) \diff x=-\int_{t_1}^{t_2} \int_\Omega \rho\diff x \diff t,$$
which is equivalent to the claim.
\end{proof}

We are now in position to prove Theorem~\ref{th long time}.

\begin{proof}[Proof of Theorem~\ref{th long time}]
Let $\rho_0 \in L^1(\O)$ be fixed.

\step{1}{Existence}

For $T>0$, we let $(\rho^T,\varphi^T)$ denote a solution of \eqref{MFGSystFiniteHorizon} with $T>0$, with initial datum $\rho_0$ for $\rho$ and with final datum $\psi^T$ for $\varphi$, where $(\psi^T)_{T > 0}$ is any family of non-negative functions, bounded in $L^\infty(\Omega) \cap H^1_0(\Omega)$.

Recall that, by Proposition~\ref{PropHJB}, $\norm{\varphi^T}_{L^\infty((0,T) \times \O)}$ is bounded independently of $T$. Let $0 < T_1 < T_2$ be fixed. Lemma~\ref{lem est phi} implies that, as soon as $T > T_2$, $\varphi^T$ is bounded in $L^2((T_1,T_2);H^2(\O))$ independently of $T>0$. Moreover, because $\partial_t\varphi^T \in L^2((T_1,T_2) \times \O)$ owing to Proposition~\ref{PropHJB}, we can apply Aubin--Lions Lemma to the sequence $(\varphi^T)_{T>0}$ to get that, up to extraction, it converges strongly in $L^2_{loc}((0,+\infty);H^1(\O))$ to some limit $\varphi_\infty$. Up to another extraction, we ensure that the convergence of $\varphi^T, \nabla \varphi^T$ also holds pointwise.

Using Aubin--Lions Lemma for the sequence $(\rho^T)_{T>0}$ as in the proof of Proposition~\ref{PropContinuityFP}, we find that, up to another extraction, it converges strongly to a limit $\rho_\infty$ in $L^2((T_1,T_2)\times\O)$ and weakly in $L^2((T_1,T_2);H^1_0(\O))$. The solutions $(\rho^T, \varphi^T)$ are associated with a bounded vector field $V_T$, which will converge weakly-$\ast$ in $L^\infty$ to a vector field $V_\infty$. Using the same arguments as in the proof of Theorem~\ref{MainTheoFixedT}, Claim~\ref{ClaimClosedGraph}, we can pass to the limit $T\to +\infty$ in the equation to find that the pair $(\rho_\infty, \varphi_\infty)$ solves \eqref{MFGSystInfiniteHorizon}.

\step{2}{Long-time behavior of $\rho$}

Let $(\rho,\varphi)$ be a solution of \eqref{MFGSystInfiniteHorizon}, as built in the previous step. The integral version of Lemma~\ref{lem rho phi}, which is valid for $(\rho^T,\varphi^T)$, also applies to $(\rho,\varphi)$ at the limit, and we have 
$$
\frac{\diff}{\diff t} \int_\O \rho(t,x)\varphi(t,x)\diff x \leq \frac{-1}{\sup \varphi}\left(\int_\O\rho(t,x)\varphi(t,x)\diff x\right),
$$
hence, for all $t \geq 0$, we have
$$
\int_\O\rho(t,x)\varphi(t,x)\diff x \leq \left(\int_\O\rho_0(x)\varphi(0,x)\diff x\right)e^{-\frac{1}{\sup \varphi} t}.
$$
Moreover, using the fact that $t\mapsto \int_\O \rho(t,x)\diff x$ is non-increasing, we get, integrating the relation from Lemma~\ref{lem rho phi},
$$
\int_\O \rho(t,x)\diff x \leq \int_{t-1}^t\int_\O \rho(\tau,x)\diff x d\tau \leq \int_\O\rho(t-1,x)\varphi(t-1,x)\diff x,
$$
from which we get that there are $\alpha,\beta>0$ such that
$$
\int_\O\rho(t,x)\diff x \leq \beta e^{-\alpha t}.
$$
Now, let us denote $u(t) := \int_\O \rho^2(t,x)\diff x$. This is well defined for all $t>0$. We have
$$
u^\prime(t) = -2\nu \int_\O \vert \nabla \rho \vert^2 - 2 \int_\O \rho V\cdot\nabla \rho,
$$
and, using Young's inequality, we get that there is $\delta>0$ (depending on $\sup \kappa$ and $\nu$) such that
$$
u^\prime - 2 \delta u \leq 0.
$$
Hence
$$
\int_\O \rho^2(t,x)\diff x \leq \left(\int_\O \rho^2(1,x) \diff x  \right)e^{2 \delta (t-1)}. 
$$
Now, let $\theta \in (0,1)$ be close enough to $1$ so that $\alpha\theta > \delta (1-\theta)$. Let $p_\theta := \theta + 2(1-\theta) >1$. By classical interpolation arguments on $L^p$ spaces, one has
$$
\norm{ \rho(t,\cdot) }_{L^{p_\theta}} \leq \norm{ \rho(t,\cdot) }_{L^{1}}^{\theta} \norm{ \rho(t,\cdot) }_{L^{2}}^{1-\theta}  \leq A e^{-(\alpha\theta - \delta(1-\theta))t},
$$
where $A = \beta^\theta e^{-(1-\theta)\delta} \norm{\rho(1,\cdot)}_{L^2}^{1-\theta}$. Now that we have that the $L^{p_\theta}$ norm of $\rho(t,\cdot)$ goes to zero as $t$ goes to $+\infty$, Corollary~\ref{reg FP} gives us that the $L^{\infty}$ norm of $\rho(t,\cdot)$ also goes to zero when $t$ goes to $+\infty$.

\step{3}{Long-time behavior of $\varphi$}

We now turn to the convergence of $\varphi$ as $t \to +\infty$.
Let $(t_n)_{n\in \N}$ be a sequence of positive real numbers diverging to $+\infty$. Define 
$$
\varphi_n(t,x) := \varphi(t+t_n,x).
$$
Then, $\varphi_n$ solves
$$
-\partial_t\varphi_n - \nu \Delta \varphi_n + \kappa(\rho(t+t_n,x))\vert \nabla \varphi_n\vert - 1 =0,\quad t>-t_n, \ x\in \O.
$$
Using the same estimates as in the first step, we find that, up to a subsequence, $\varphi_n$ converges to some $\overline{\varphi}(t,x)$ in the $L_{loc}^2(H^1)$ sense, that satisfies
$$
-\partial_t\overline{\varphi}- \nu \Delta \overline{\varphi} + \kappa(0)\vert \nabla \overline{\varphi}\vert - 1 =0,\quad t \in \R, \ x\in \O,
$$
where we have used the uniform convergence $\rho(t, \cdot) \to 0$ as $t \to +\infty$ from the previous step in order to get the convergence of $\kappa(\rho(t + t_n, x))$ to $\kappa(0)$ as $n \to +\infty$. We now want to prove $\overline{\varphi} = \Psi$. From the boundedness of $\varphi$, the function $\overline{\varphi}$ is also bounded.

Let $T>0$ be fixed. Let $u_T, v_T$ be the solutions of 
\begin{equation}\label{eq 0}
-\partial_t u- \nu \Delta u + \kappa(0)\vert \nabla u\vert - 1 =0,\quad t \in (0,T), \ x\in \O,
\end{equation}
with homogeneous Dirichlet boundary conditions and final data $u_T(T,\cdot) = 0$ and $v_T(T,\cdot) =\Phi\vert_{\O} + M$, where $M \geq  \overline{\varphi}$ and $\Phi\vert_{\O} \geq 0$ is the restriction to $\O$ of the solution of the torsion equation $-\nu\Delta \Phi  = 1$ in $\O^+$ (with $\O^+$ a domain that contains $\O$, say $\O^+ := \O +B_1$) with Dirichlet boundary conditions. We recall that the existence of $u_T,v_T$ is guaranteed by Proposition~\ref{PropHJB}.

The parabolic comparison principle, Proposition~\ref{prop cp}, implies that, for every $T>0$,
$$
u_T(t,\cdot)\leq \overline{\varphi}(t,\cdot)\leq v_T(t,\cdot), \quad \text{ for }\  t\in (0,T).
$$
Let us prove that $u_T, v_T$ converge to $\Psi$, the stationary solution of \eqref{eq 0}, as $T$ goes to $+\infty$. To get this, let us show that the sequences of functions $(u_T)_{T>0}$ and $(v_T)_{T>0}$ are non-decreasing and non-increasing respectively, in the sense that $u_{T} \leq u_{T+h}$ and $v_{T} \geq v_{T+h}$ on $(0,T)\times \O$ for every $h\in (0,T)$. 

Let $T>0$ be fixed and let $h\in (0,T)$. Because \eqref{eq 0} is autonomous, $u_{T+h}$ and $u_{T}$ are both solutions of \eqref{eq 0} on $(0,T)\times \O$, with final data $u_{T+h}(T,\cdot)$ and $u_{T}(T,\cdot)=0$ respectively.
However, because $u_{T+h}(t,\cdot)\geq 0$ for $t \in (0,T+h)$ (as recalled in Proposition~\ref{PropHJB}), we have $u_{T+h}(T,\cdot)\geq u_T(T,\cdot)$. To phrase it differently, $u_{T+h}$ and $u_{T}$ are solutions of the same equation with ordered final data, hence, we can apply the comparison principle Proposition \ref{prop cp} to find that $u_{T+h}\geq u_T$ on $(0,T)\times \O$.

Similarly , we have that $v_{T+h}$ and $v_{T}$ solve \eqref{eq 0} on $(0,T)\times \O$, with final data
$v_{T+h}(T,\cdot)$ and $v_{T}(T,\cdot) = \Phi\vert_{\O} + M$. By a standard comparison principle, we have that $v_{T+h}\leq \Phi\vert_{\O} +M$. Therefore, we can apply the parabolic comparison principle Proposition~\ref{prop cp} to get that $v_{T+h}\leq v_{T}$ on $(0,T)\times \O$.

Therefore, owing to theses monotonicities, the sequences $(u_T)_{T>0}$ and $(v_T)_{T>0}$ converge a.e.\ as $T$ goes to $+\infty$ to functions that do not depend on the $t$ variable (this last fact comes from the equality 
$u_{T}(\cdot,\cdot) = u_{T+h}(\cdot+h,\cdot)$, which is true because \eqref{eq 0} is autonomous and because the solutions are unique). Moreover, arguing as in the first step, we have that these limiting functions are solutions of \eqref{eq 0}. The only stationary solution of~\eqref{eq 0} being $\Psi$, we get that $\overline\varphi(t, \cdot) = \Psi$ for every $t$. We have thus proven that
$$
w_n(t,x) := \varphi(t+t_n,x)-\Psi(x) \underset{n \to +\infty}{\longrightarrow} 0,
$$
in the $L^2_{loc}(H^1)$ sense.

Let us prove that this convergence is actually uniform. To this aim, observe that $w_n$ is a weak solution of
\begin{equation*}
-\partial_t w_n = \nu \Delta w_n -\kappa(\rho(\cdot+t_n,\cdot)) z_n\cdot \nabla w_n +(\kappa(0)-\kappa(\rho(\cdot+t_n,\cdot)))\vert \nabla \Psi\vert,
\end{equation*}
where $z_n := \frac{\nabla \varphi_n + \nabla \Psi}{\vert \nabla \varphi_n\vert + \vert \nabla \Psi\vert}$ is bounded. Then, for every $t_1, t_2$ such that $t_2 + 1 < t_1 < t_2 + 2$, using Corollary~\ref{reg HJB}, we find that
$$
\norm{ w_n(t_2,\cdot)}_ {L ^\infty} \leq C\left(\norm{w_n(t_1,\cdot)}_ {L ^2} + \norm{ (\kappa(0)-\kappa(\rho(\cdot+t_n,\cdot)))\abs{\nabla \Psi}  }_ {L^\infty((t_1,t_2)\times \O))}\right). 
$$
Integrating this for $t_1\in (t_2 + 1, t_2 + 2)$, we find
$$
\norm{ w_n(t_2,\cdot)}_ {L ^\infty} \leq C\left(\norm{w_n}_ {L^2((t_2 + 1, t_2 + 2) \times \Omega)} + \norm{ (\kappa(0)-\kappa(\rho(\cdot+t_n,\cdot)))\abs{\nabla \Psi}  }_ {L^\infty((t_1,t_2)\times \O))}\right). 
$$
Because $w_n$ goes to zero in the $L_{loc}^2(H^1)$ sense and $\vert \nabla \Psi\vert$ is bounded, observing that $\vert\kappa(\rho(\cdot+t_n,\cdot))-\kappa(0)\vert$ converges uniformly to zero (this comes from the uniform convergence to zero of $\rho$ from Step 2), we obtain that $w_n$ goes to zero uniformly, whence
$$
\varphi(t,x)\underset{t\to+\infty}{\longrightarrow} \Psi(x)
$$
in the $L^\infty$ sense.
\end{proof}

\subsection{Improved convergence results}

In the previous section, Theorem~\ref{th long time} proved the existence of solutions $(\rho, \varphi)$ to the MFG system with infinite time horizon \eqref{MFGSystInfiniteHorizon} and characterized the asymptotic behavior of any such solution by providing uniform convergence $\rho_t\to 0$ and $\varphi_t\to \Psi$. We want here to improve this result in two ways: first, we will prove that this convergence is actually exponential (in what concerns $\varphi$ this requires a very small extra assumption on the function $\kappa$); second, we will prove that the convergence of $\varphi(t, \cdot)$ to $\Psi$ as $t \to +\infty$, in addition to being uniform, is also a strong convergence in $H_0^1(\Omega)$. This last result is natural to evoke, because of the role played by $\nabla\varphi$ in the dynamics.

\begin{proposition}
\label{PropExponentialConvergence}
Suppose that the function $\kappa:\R_+\to\R_+$ is H\"older continuous. Then, there exist constants $C,\alpha>0$ (depending on $\kappa$, $\nu$, and $\Omega$), such that we have, for any $(t,x)\in [0,+\infty)\times\Omega$,
$$\abs{\rho(t,x)}+\abs{\varphi(t,x)-\Psi(x)}\leq Ce^{-\alpha t}.$$
\end{proposition}

\begin{proof}
The exponential convergence of $\rho$ to $0$ is indeed part of the proof of Theorem~\ref{th long time}, since we proved that, for $p$ close to $1$, the $L^p$ norm of $\rho_t$ tends exponentially to $0$, and we then used the parabolic regularization estimate $\norm*{ \rho_t}_{L^\infty}\leq C\norm*{ \rho_{t-1}}_{L^p}$.

Thanks to the assumption that $\kappa$ is H\"older continuous, up to modifying the coefficient in the exponent, we obtain $\abs{K(t,x)-\kappa(0)}\leq Ce^{-\alpha t}$, where $K(t,x)=\kappa(\rho(t,x))$. 

We need now to discuss the exponential convergence of $\varphi$. Let us fix a time $t_1$ and define 
$$a_\pm:=1\pm 3C\norm*{\nabla\psi}_{L^\infty}e^{-\alpha t_1},\quad \Psi_{\pm}:=a_\pm \Psi\pm e^{-\alpha t_1}.$$
We will use a comparison principle between $\varphi$ and $\Psi_\pm$. The functions $\Psi_\pm$ solve
$$-\partial_t \Psi_\pm -\nu\Delta\Psi_\pm+\kappa(0)|\nabla\Psi_\pm|-a_\pm=0,$$
where the time-derivative term is actually $0$ since they are functions of the $x$ variable only.
If we set $v_\pm:=\varphi-\Psi_\pm$, the functions $v_\pm$ solve a linear PDE of the form
$$-\partial_t v_\pm -\nu\Delta v_\pm+w_\pm\cdot\nabla v_\pm \pm 3C\norm*{\nabla\psi}_{L^\infty}e^{-\alpha t_1}+(K(t,x)-\kappa(0))a_\pm|\nabla\Psi|=0,$$
where the vector fields $w_\pm$ are such that $|w_\pm(t,x)|\leq K(t,x)$.
In particular, if we note that, for $t_1$ large enough, we have $0\leq a_\pm\leq 2$, we have $(K(t,x)-\kappa(0))a_\pm|\nabla\Psi|\leq 2C\norm*{\nabla\psi}_{L^\infty}e^{-\alpha t_1}$. Hence, for $v_+$ we have
$$-\partial_t v_+ -\nu\Delta v_++w_+\cdot\nabla v_+< 0$$
and for $v_-$
$$-\partial_t v_- -\nu\Delta v_-+w_-\cdot\nabla v_-> 0.$$
Let us look now at the boundary conditions of $v_\pm$ relative to the parabolic domain $[t_1,t_2]\times \Omega$. If $t_2$ is large enough, using the uniform convergence $\varphi_t\to \Psi$, we can infer $v_+(t_2,x)<0$ for every $x\in\Omega$. Moreover, we also have $v_+(t,x)<0$ for every $t$ and every $x\in\partial\Omega$. The inequalities are opposite for $v_-$, i.e. we have $v_-(t,x)>0$ for $t=t_2$ or $x\in \partial\Omega$. This implies, by the maximum principle in \cite{Aronson1967Local} (see \cite[Theorem~1]{Aronson1967Local}, adapted to this backward equation, and using again the version with the inequality presented at the end of the proof, page~98), the inequalities $v_+(t_1,x)\leq 0\leq v_-(t_1,x)$, i.e. 
$$(1- 3C\norm*{\nabla\psi}_{L^\infty}e^{-\alpha t_1})\Psi- e^{-\alpha t_1}\leq \varphi\leq (1+ 3C\norm*{\nabla\psi}_{L^\infty}e^{-\alpha t_1})\Psi+ e^{-\alpha t_1}.$$
This shows $\norm*{\varphi_t-\Psi}_{L^\infty}\leq C e^{-\alpha t_1}$, for a new constant $C$.
\end{proof}

We can now pass to the following statement, which proves the convergence of the gradient of $\varphi$.
\begin{proposition}
\label{prop:VarphiConvergesInH1}
Let $(\rho, \varphi)$ be a solution to the Mean Field Game system with infinite time horizon \eqref{MFGSystInfiniteHorizon}. Then
\[
\varphi(t, \cdot) \underset{t \to +\infty}{\longrightarrow} \Psi
\]
in $H_0^1(\Omega)$.
\end{proposition}

\begin{proof}
We first observe that, by Lemma~\ref{lem est phi}, the family of functions $(\varphi(t, \cdot))_{t \geq 0}$ is bounded in $H_0^1(\Omega)$. This, together with the uniform convergence to $\Psi$, implies that one has the weak convergence $\varphi(t, \cdot) \xrightharpoonup{} \Psi$ in $H_0^1(\Omega)$ as $t \to +\infty$. 

In order to conclude the proof, it suffices to show $\norm{\varphi(t,\cdot)}_{H_0^1(\Omega)} \to \norm{\Psi}_{H_0^1(\Omega)}$ as $t \to +\infty$. Since $\varphi(t,\cdot) \xrightharpoonup{} \Psi$ in $H_0^1(\Omega)$ as $t \to \infty$, one has
\[\norm{\Psi}_{H_0^1(\Omega)}^2 \leq \liminf_{t \to +\infty} \norm{\varphi(t,\cdot)}_{H_0^1(\Omega)}^2.\]
Let us argue by contradiction and assume that there exists $\varepsilon > 0$ and an increasing sequence $(s_n)_{n \in \mathbb N}$ with $s_n \to +\infty$ as $n \to +\infty$ such that
\begin{equation}
\label{eq:LowerBoundPhiSn}
\norm*{\varphi(s_n)}_{H_0^1(\Omega)}^2 \geq \norm{\Psi}_{H_0^1(\Omega)}^2 + 2 \varepsilon
\end{equation}
for every $n \in \mathbb N$. Recall that, from \eqref{a int} in the proof of Lemma~\ref{lem est phi}, there exists $C > 0$ depending only on $\sup\varphi$, $\sup\kappa$, $\nu$, and $\abs{\Omega}$ such that
\[
\norm{\varphi(t,\cdot)}_{H_0^1(\Omega)}^2 \geq \norm{\varphi(s_n)}_{H_0^1(\Omega)}^2 e^{- C (t - s_n)} - C (1 - e^{-C(t - s_n)})
\]
for every $n \in \mathbb N$ and $t \in (s_n, s_n + 1)$. Combining this with \eqref{eq:LowerBoundPhiSn}, one obtains that there exists $\delta \in (0, 1)$ depending only on $\norm{\Psi}_{H_0^1(\Omega)}$, $\varepsilon$, and $C$ such that
\begin{equation}
\label{eq:LowerBoundPhiT}
\norm*{\varphi(t,\cdot)}_{H_0^1(\Omega)}^2 \geq \norm{\Psi}_{H_0^1(\Omega)}^2 + \varepsilon
\end{equation}
for every $n \in \mathbb N$ and $t \in (s_n, s_n + \delta)$. By Lemma~\ref{lem est phi}, there exists a constant $C^\prime > 0$ depending only on $\sup\varphi$, $\sup\kappa$, $\nu$, $\abs{\Omega}$, and $\Psi$ such that
\[
\norm{\varphi}_{L^2((s_n, s_n + \delta); H^2(\Omega))}^2 \leq C^\prime \qquad \text{ for every } n \in \mathbb N.
\]
In particular, for every $n \in \mathbb N$, there exists $t_n \subset (s_n, s_n + \delta)$ such that $\norm{\varphi(t_n)}_{H^2(\Omega)}^2 \leq C^\prime /\delta$. Hence $(\varphi(t_n))_{n \in \mathbb N}$ is bounded in $H^2(\Omega)$ and thus, up to extracting subsequences (which we still denote by $(s_n)_{n \in \mathbb N}$ and $(t_n)_{n \in \mathbb N}$ for simplicity), $(\varphi(t_n))_{n \in \mathbb N}$ converges strongly in $H_0^1(\Omega)$ as $n \to \infty$. Since $\varphi(t,\cdot) \xrightharpoonup{} \Psi$ as $t \to +\infty$, the strong limit of $(\varphi(t_n))_{n \in \mathbb N}$ in $H_0^1(\Omega)$ is necessarily $\Psi$, and thus, in particular, $\norm{\varphi(t_n)}_{H_0^1(\Omega)}^2 \to \norm{\Psi}_{H_0^1(\Omega)}^2$ as $n \to +\infty$. This, however, contradicts  \eqref{eq:LowerBoundPhiT}, and establishes the desired result.
\end{proof}

\appendix

\section{Regularizing effects of parabolic equations}
\label{AnnexeA}

This appendix is concerned with the regularizing properties of a class of parabolic equations including both the Fokker--Planck and the Hamilton--Jacobi--Bellman equations we consider in this paper. More precisely, we consider the increase of the exponent $p$ of the $L^p$ integrability in space of the solution of the system. As recalled in the introduction, the computations and results presented here are very similar to those from the appendix of \cite{Cardaliaguet2013Long1}, the main difference lying in the boundary condition. The main result of this appendix is the following.

\begin{proposition}\label{prop L infty appendix}
Let $T\in (0,+\infty]$. Let $V, F, f, g, u \in C^\infty((0,T)\times \O)$ with $V, g \in L^\infty((0, T) \times \O)$, $u \geq 0$, $u=0$ on $\partial \O$, such that
\begin{equation}\label{general parabolic}
\partial_t u - \nu \Delta u + \nabla \cdot (u V) + \nabla \cdot F + f + g \cdot \nabla u \leq 0,\quad \text{on} \quad (0,T) \times \Omega.
\end{equation}
Then, for every $p>1$, every number $a \in (0, 1)$ and $t_1,t_2$ such that $0<t_1<t_2<T$ and $a<\abs{t_1-t_2}<a^{-1}$, there is $C>0$, depending only on $p,a, \norm{V}_{L^\infty}, \norm{g}_{L^{\infty}}$ such that
$$
\norm{u(t_2,\cdot)}_{L^{\infty}} \leq C \left( \norm{u(t_1,\cdot)}_{L^p} + \norm{F}_{L^\infty((t_1,t_2)\times\O)}
+ \norm{f}_{L^\infty((t_1,t_2)\times\O)}\right).
$$
The same result is true omitting the assumption $u\geq 0$ if the PDE \eqref{general parabolic} is satisfied as an equality instead of an inequality.
\end{proposition}
The proof follows a standard method based on Moser's iterations that will be detailed here. This appendix is included for completeness: the experienced reader will recognize well-known computations, which are simplified in this setting thanks in particular to the Dirichlet boundary conditions we use.
\begin{proof}
Let $u$ be as in the proposition. For $k>1$, we define
$$
m_k(t) := \int_{\Omega} u^k(t,x)\diff x.
$$
We also define $\alpha := \frac{2^\star}{2} = \frac{n}{n-2}$ if $n>2$ (here $2^\star$ is the Sobolev exponent in dimension $n$). When $n=1, 2$ we set $\alpha := 2$ (but any number larger than $1$ and smaller than $+\infty$ could be used here). Moreover, we set
$$
M := \norm{F}_{L^\infty((t_1,t_2)\times\O)}+ \norm{f}_{L^\infty((t_1,t_2)\times\O)}
$$

\step{1}{$L^p$ estimates}

Let us start with proving that, for $k_0>1$, there is $C>0$ depending on $k_0$ and on the $L^\infty$ norms of $V,g$, such that, for every $k>k_0>1$,
\begin{equation}\label{eq m}
\frac{\diff}{\diff t}(m_k e^{-Ck^2 t}) +  \frac{1}{C} m_{\alpha k}^{\frac{1}{\alpha}} e^{-Ck^2 t} \leq Ce^{-Ck^2 t} k^2 M^k.
\end{equation}
In order to do so, we differentiate $m_k$ with respect to $t$, to get
\begin{multline*}
m_k^\prime(t) \leq k\int_{\O}(\nu\Delta u - \nabla\cdot(u V)-g\cdot \nabla u - \nabla \cdot F - f ) u^{k-1} \\
 \leq - k(k-1) \nu \int_{\O} \vert \nabla u \vert^2 u^{k-2} + k(k-1) \int_{\O}(V\cdot \nabla u) u^{k-1} - k\int_\O (g \cdot \nabla u) u^{k-1} \\ + k(k-1)\int_{\O} (F\cdot \nabla u) u^{k-2} - k\int_{\O}f u^{k-1}.
\end{multline*}
Now, owing to Young's inequality, we can find $C_1,C_2,C_3>0$ depending only on  $\norm{V}_{L^\infty}$, $\norm{g}_{L^{\infty}}$, $k_0$, $\nu$ such that
$$
m_k^\prime(t) \leq 
  - C_1 k^2\int_{\O} \vert \nabla u \vert^2 u^{k-2} + C_2k^2 \int_{\O}u^{k} + C_3 k^2 \int_{\O} \vert F\vert^2  u^{k-2} + k \int_{\O}\vert f \vert u^{k-1}
$$
(note that we replaced the coefficient $k(k-1)$ with $k^2$, as these two numbers are equivalent up to multiplicative constants as far as $k > k_0>1$).
Moreover, thanks again to a Young inequality, we have
$$
 \vert F\vert^2  u^{k-2} \leq \frac{2}{k} \vert F\vert^k  + \frac{k-2}{k}   u^{k} \quad\text{ and }\quad  \vert f\vert u^{k-1} \leq \frac{1}{k}\vert f \vert^k + \frac{k-1}{k}u^{k}. 
$$
Therefore, up to increasing $C_2,C_3$, we get
\begin{multline*}
m_k^\prime(t) \leq 
  - C_1 k^2\int_{\O} \vert \nabla u \vert^2 u^{k-2} + C_2k^2 \int_{\O}u^{k} + C_3 k^2 \int_{\O} \vert F\vert^k  + k \int_{\O}\vert f \vert^k \\
  \leq 
  - C_1 k^2\int_{\O} \vert \nabla u \vert^2 u^{k-2} + C_2 k^2 \int_{\O}u^{k} + C_3 k^2 M^k.
\end{multline*}
Now, owing to the Gagliardo--Nirenberg--Sobolev inequality, we have, for some $C_4>0$,
$$
 k^2\int_{\O} \vert \nabla u \vert^2 u^{k-2} = 4\int_{\O}\vert\nabla( u^\frac{k}{2} )\vert^2 \geq C_4 \left(\int_{\O}u^{k\alpha}\right)^{\frac{1}{\alpha}} = C_4 m_{k\alpha}^\frac{1}{\alpha}.
$$
Hence
$$
m_k^\prime(t) + C_1C_4 m_{\alpha k}^{\frac{1}{\alpha}}\leq 
   C_2k^2 m_k +C_3 k^2 M^k.
$$
Let us denote $C:= \max\{\frac{1}{C_1C_4},C_2,C_3\}$. Then, the above equation gives us
$$
m_k^\prime(t) -Ck^2 m_k + \frac{1}{C} m_{\alpha k}^{\frac{1}{\alpha}}\leq C k^2 M^k,
$$
which we can rewrite in order to get \eqref{eq m}.

\step{2}{Estimates on $m_{\alpha k}$}

We show in this step that, for $k> k_0>1$ and for $0<t_1<t_2<T$, we have
\begin{equation}\label{step 2}
m_{\alpha k}^{\frac{1}{\alpha}}(t_2) \leq  e^{C\alpha k^2(t_2-t_1)}\frac{1}{t_2-t_1}e^{C k^2t_ 2}\int_{t_1}^{t_2}m_{\alpha k}^{\frac{1}{\alpha}}(s)e^{-Ck^2s}\diff s  + e^{C\alpha k^2(t_2-t_1)}M^k,
\end{equation}
for some $C$ depending on $k_0$ and on the $L^\infty$ norms of $V,g$.

The relation \eqref{eq m} provides
$$
\frac{\diff}{\diff t}(m_k e^{-Ck^2 t})  \leq Ce^{-Ck^2 t} k^2 M^k.
$$
Let us take $s \in (t_1,t_2)$. We integrate the above inequality for $t\in (s,t_2)$ to get:
$$
m_k(t_2) e^{-Ck^2  t_2}  \leq m_k(s) e^{-Ck^2  s} + CM^k\int_s^{t_2} e^{-Ck^2 t}\diff t\leq m_k(s) e^{-Ck^2  s} + M^k  e^{-C k^2 s}.
$$
Taking the power $\frac{1}{\alpha}<1$ and using its subadditivity yields
$$
m_k^{\frac{1}{\alpha}}(t_2) e^{-\frac{C}{\alpha}k^2  t_2}  \leq  m_k^{\frac{1}{\alpha}}(s)e^{-\frac{C}{\alpha}k^2 s}   + M^{\frac k\alpha}e^{-\frac{C}{\alpha}k^2  s} .
$$
We replace $k$ by $\alpha k$ so as to re-write the above inequality as
$$
m_{\alpha k}^{\frac{1}{\alpha}}(t_2) e^{-C\alpha k^2  t_2}  \leq  m_{\alpha k}^{\frac{1}{\alpha}}(s)e^{-C\alpha k^2 s}   + M^ke^{-C\alpha k^2  s} .
$$
We multiply by $e^{Ck^2s(\alpha-1)}$ and integrate this for $s  \in (t_1,t_2)$ in order to obtain
$$
m_{\alpha k}^{\frac{1}{\alpha}}(t_2) \int_{t_1}^{t_2}e^{C\alpha k^2  (s-t_2)}e^{-Ck^2s}\diff s  \leq  \int_{t_1}^{t_2}m_{\alpha k}^{\frac{1}{\alpha}}(s)e^{-Ck^2 s}\diff s   + M^k\int_{t_1}^{t_2}e^{-Ck^2  s} \diff s .
$$
We then use $e^{C\alpha k^2  (s-t_2)}\geq e^{C\alpha k^2  (t_1-t_2)}$ in order to obtain
$$
m_{\alpha k}^{\frac{1}{\alpha}}(t_2)   \leq e^{C\alpha k^2  (t_2-t_1)} \left(\int_{t_1}^{t_2}e^{-Ck^2s}\diff s\right)^{-1}\int_{t_1}^{t_2}m_{\alpha k}^{\frac{1}{\alpha}}(s)e^{-Ck^2 s}\diff s   + e^{C\alpha k^2  (t_2-t_1)}M^k
$$
and finally we use $\int_{t_1}^{t_2}e^{-Ck^2s}\diff s\geq (t_2-t_1)e^{-Ck^2t_2}$, which provides the desired inequality.

\step{3}{Higher integrability estimates}

Let us now show that, for $k>k_0$, there is $C>0$ (depending on the same quantities as in the previous steps), such that
\begin{equation}\label{step 3}
m_{\alpha k}^{\frac{1}{\alpha k}}(t_2)  \leq  \frac{e^{Ck(t_2-t_1)}}{(C^{-1}|t_2-t_1|)^{1/k}}\left(m_k(t_1)  +   M^k\right)^{\frac{1}{k}}.
\end{equation}
First of all, integrating \eqref{eq m} for $t\in (t_1,t_2)$, and discharging the final value $m_k(t_2)e^{-Ct_2}$, we obtain
$$
\frac{1}{C}\int_{t_1}^{t_2} m_{\alpha k}^{\frac{1}{\alpha}}(t) e^{-C k^2 t}\diff t \leq m_k(t_1) e^{-Ck^2 t_1} +  M^k\int_{t_1}^{t_2}Ck^2e^{-Ck^2 t}\diff t\leq e^{-Ck^2t_1}\left(M^k+m_k(t_1)\right).
$$
Combining this with \eqref{step 2}, we get
$$
m_{\alpha k}^{\frac{1}{\alpha}}(t_2)  \leq  e^{C\alpha k^2(t_ 2-t_1)}\left(C\frac{e^{Ck^2(t_2-t_1)}}{t_2-t_1}(m_k(t_1)  +M^k) + M^k \right).
$$
Up to enlarging the constant $C$ and using $0 < t_2-t_1<a^{-1}$, we can write the above inequality in a simpler form, i.e.
$$
m_{\alpha k}^{\frac{1}{\alpha}}(t_2)  \leq  \frac{e^{C(\alpha+1)k^2(t_2-t_1)}}{C^{-1}(t_2-t_1)}\left(m_k(t_1)  +   M^k\right),
$$
hence, \eqref{step 3} holds true, after raising to the power $1/k$ and including $\alpha+1$ in the constant $C$.

\step{4}{Iterations}

We conclude the proof in this step by proving that, for $p,t_1,t_2$ as in the statement of the proposition, there is $C>0$ such that
\begin{equation}\label{step 4}
\norm{u(t_2,\cdot)}_{L^{\infty}} \leq C ( \norm{u(t_1,\cdot)}_{L^p} + \norm{F}_{L^\infty}
+ \norm{f}_{L^{\infty}}).
\end{equation}
We denote
$$
s_n := t_2 - \frac{t_2-t_1}{(2\alpha)^n}, \ k_n := \alpha^n p, \ \beta_n := \frac{e^{Ck_n(s_{n+1}-s_n)}}{(C^{-1}(s_{n+1} - s_{n}))^{\frac{1}{k_n}}},
$$
and
$$
a_n := m_{k_n}^{\frac{1}{k_n}}(s_n),\quad \tilde a_n:=\max\{a_n, M\}.
$$
Then, \eqref{step 3} gives us that
$$
a_{n+1} \leq \beta_n (a_n^{k_n} +M^{k_n})^{\frac{1}{k_n}}\leq \beta_n 2^{\frac{1}{k_n}}\tilde a_n.
$$
Hence, up to replacing the constant $C$ with a larger one so as to suppose $\beta_n 2^{\frac{1}{k_n}}\geq 1$, we find 
$$
\tilde a_{n+1} \leq \beta_n 2^{\frac{1}{k_n}}\tilde a_n.
$$
We observe that we have $\prod_{n=0}^{+\infty} \beta_n 2^{\frac{1}{k_n}} < +\infty$ as a consequence of the logarithmic estimate
\[
\sum_{n=0}^{+\infty} \log(\beta_n 2^{\frac{1}{k_n}}) \leq \sum_{n=0}^{+\infty} Ck_n\frac{t_2-t_1}{(2\alpha)^n}+\frac{1}{k_n}(\log 2+n\log(2\alpha)-\log(t_2-t_1)+\log C) < +\infty.
\]
Therefore, we obtain
$$
\max\{ \lim_{n\to +\infty} a_n , M\} \leq \left(\prod_{n=0}^{+\infty} \beta_n 2^{\frac{1}{k_n}}\right)\max\{a_0 , M \}\leq C(a_0+M).
$$
Hence, thanks to $\lim_{n\to +\infty}a_n = \norm{u(t_2,\cdot)}_{L^{\infty}}$, we obtain \eqref{step 4}. This concludes the proof.
\end{proof}

\begin{corollary}\label{reg FP}
Let $T\in (0,+\infty]$. Let $V \in L^{\infty}((0,T)\times \O)$. Let $u\in L^1((0,T)\times\O)$ be a positive distributional solution of
$$
\partial_t u - \nu \Delta u -\nabla\cdot(u V) \leq  0,\quad \text{ on } \ (0,T)\times \O,
$$
satisfying the following mild regularity assumption: $u$ is obtained as a measurable curve $(u_t)_t$ of functions of the $x$ variable, which is such that $t\mapsto\int_\O \eta(x) u_t(x) \diff x$ is continuous in time for every $\eta\in C^\infty(\O)$ (note that we do not restrict to $\eta\in C^\infty_c(\O)$).
Then, for every $p>1$ and $a \in (0, 1)$, there is $C>0$, depending only on $p$, $a$, $\norm{V}_{L^\infty}$, such that we have
$$
\norm{u(t_2,\cdot)}_{L^{\infty}} \leq C  \norm{u(t_1,\cdot)}_{L^p}
$$
for every $0<t_1<t_2<T$ with $a<\abs{t_2-t_1}<a^{-1}$.
\end{corollary}

\begin{proof}
To prove this estimate the only important point is to regularize the equation so as to apply Proposition~\ref{prop L infty appendix}. In order for the proof to be self-contained, we detail a two-step approximation argument.

We convolve the equation by an approximation of the identity and to apply Proposition~\ref{prop L infty appendix}. However, convolving will not preserve the Dirichlet boundary conditions, so we first have to extend $u$ by zero on a bigger set.

We define $\O^+$ to be a open bounded regular set such that $\O + B_1 \subset \O^+$, where $B_1$ is unit ball in $\R^N$. 

We define $u^+(t,x) := u(t,x)$ if $x\in \O$, and $u^+(t,x) = 0$ elsewhere.
Let $\eta_\e(x)$ be an approximation of the identity whose support is included in $B_1$. We define $u_\e := u^+\star \eta_\e$ (here, $\star$ is the convolution in space only). It is a function which is smooth in $x$ and continuous in $t$. We then convolve in time as well, taking $\chi_\delta(t)$ an approximation of the identity whose support is included in $\R_+$. Defining $u_{\e,\delta}:=\chi_\delta\star u_\e$ we have now a function which is smooth in time and space. It satisfies, in the classical sense,
$$
\partial_t u_{\e,\delta} -\nu \Delta u_{\e,\delta} -\nabla \cdot( u_{\e,\delta} V_{\e,\delta} )\leq 0, \quad \text{ for } \ t\in (0,T), \ x\in \O^+,
$$
with $V_{\e,\delta} := \frac{\chi_\delta\star\eta_\e\star (u V)}{ u_{\e,\delta}} \in C^\infty$. Moreover, the $L^\infty$ norm of $V_{\e,\delta}$ is bounded independently of $\e$ and $\delta$. Then, $u_{\e,\delta}$ is positive, regular and is a (classical) subsolution of a Fokker--Planck equation with regular coefficients, hence we can apply Proposition~\ref{prop L infty appendix}. We then take the limit $\delta\to 0$, and we observe that we have
$$\norm{u_\e(t,\cdot)}_{L^p}=\lim_{\delta\to 0}\norm{u_{\e,\delta}(t,\cdot)}_{L^p}$$
for every $t$, since $u_\e$ is continuous. Then, we have 
$$\norm{u(t,\cdot)}_{L^p}=\lim_{\e\to 0}\norm{u_{\e}(t,\cdot)}_{L^p}$$
from standard properties of convolutions (with the possibility, of course, that this limit and this norm take the value $+\infty$).
\end{proof}

\begin{corollary}\label{reg HJB}
Let $T \in (0,+\infty]$.
Let $f, g\in L^\infty$ and let $u\in L^\infty((0,T);L^2(\O))\cap L^2((0,T); H^1_0(\O))$ be solution (in the weak sense) of
$$
 \partial_t u - \nu\Delta u +  f + g \cdot \nabla u = 0,\quad \text{on} \quad (0,T) \times \Omega,
$$
with Dirichlet boundary conditions and initial datum $u(0,\cdot) = u_0 \in L^2$.

Then, for every $p>1$, and $a \in (0, 1)$ there is $C>0$, depending only on $p,a,  \norm{g}_{L^{\infty}}$ such that
$$
\norm{u(t_2,\cdot)}_{L^{\infty}} \leq C \left( \norm{u(t_1,\cdot)}_{L^p}
+ \norm{f}_{L^{\infty}}\right)
$$
for every $t_1<t_2$ with $a<\abs{t_2-t_1}<a^{-1}$.
\end{corollary}

\begin{proof}
Let $f_n,g_n$ be $C^\infty$ and such that $f_n \to f$ and $g_n \to g$ in the $L^2$ norm. Assume moreover that we have $\norm{f_n}_{L^\infty}\to \norm{f}_{L^\infty}$ and $\norm{g_n}_{L^\infty}\to \norm{g}_{L^\infty}$. Let $u_n$ be the solution of
$$
 \partial_t u_n - \nu\Delta u_n +  f_n + g_n \cdot \nabla u_n = 0,\quad \text{on} \quad (0,+\infty) \times \Omega,
$$
with Dirichlet boundary condition and with initial datum $u_0^n$, where $u_0^n$ is a smooth $L^2$ approximation of $u_0$. 

Then, $u_n$ is smooth enough to apply Proposition~\ref{prop L infty appendix} to $u_n$, to get, for $p,t_1,t_2$ as in the statement of the corollary,
\begin{equation}    \label{estimate un LinftyLp}
\norm{u_n(t_2,\cdot)}_{L^{\infty}} \leq C \left( \norm{u_n(t_1,\cdot)}_{L^p} + \norm{f_n}_{L^\infty}
\right).
\end{equation}
Then, as $n$ goes to $+\infty$, $u_n$ converges (the arguments to prove this are standard and based on the weak $L^2$ convergence of $\nabla u_n$) to a solution (in the weak sense) of
$$
 \partial_t u - \nu\Delta u +  f + g \cdot \nabla u = 0,\quad \text{on} \quad (0,+\infty) \times \Omega,
$$
with Dirichlet boundary conditions and with initial datum $u_0$. The convergence is also strong in the $L^2$ sense. Because this solution is unique, it necessarily coincides with the original solution $u$ of the statement. In order to obtain the result, we need to pass to the limit the inequality \eqref{estimate un LinftyLp}. The left-hand side can easily be dealt with by semicontinuity, while for the right-hand side, we suppose $p\leq 2$ and we use strong $L^2$ convergence. Since this convergence is $L^2$ in space-time, we have convergence of the right-hand side only for a.e.\ $t_1$. Yet, using the fact that the solution $u$ is continuous in time as a function valued into $L^2(\O)$, the result extends to all $t_1$. The inequality for $p=2$ implies that with $p>2$, up to modifying the constant in a way depending on $\abs{\O}$.
\end{proof}

The reader may observe that we used different regularization strategies to prove the two above corollaries. Indeed, the linear behavior of the Fokker--Planck equation allowed to directly regularize the solution (up to modifying the drift vector field: we convolve the solution and define a new drift vector field which preserves the same $L^\infty$ bound, a trick which is completely standard for curves in the Wasserstein space, see for instance \cite[Chapter 5]{Santambrogio2015Optimal}). This is not possible for the Hamilton--Jacobi--Bellman equation. However, when uniqueness of the solution is known, regularizing the coefficients and the data of the equation is another option, and it is what we did in our last corollary.

\bigskip

\noindent\textbf{Acknowledgments.} The authors wish to thank many colleagues for useful discussions and suggestions, and in particular Alessio Porretta. Without the comments he made after a talk the second author gave on the topic of the present paper, the strategy to achieve convergence to a solution in the limit $T \to +\infty$ would have been completely different, the result less general, and the time needed to achieve it much longer.

The authors acknowledge the financial support of French ANR project ``MFG'', reference ANR-16-CE40-0015-01, and of a public grant as part of the ``Investissement d'avenir'' project, reference ANR-11-LABX-0056-LMH, LabEx LMH, PGMO project VarPDEMFG. The first author was also partially supported by the by the French IDEXLYON project Impulsion ``Optimal Transport and Congestion Games'' PFI 19IA106udl and the second author was also partially supported by the Hadamard Mathematics LabEx (LMH) through the grant number ANR-11-LABX-0056-LMH in the ``Investissement d'avenir'' project.

\bibliographystyle{abbrv}
\bibliography{Bib}

\begin{thebibliography}{10}

\bibitem{AchPor-cong}
Y.~Achdou and A.~Porretta.
\newblock Mean field games with congestion.
\newblock {\em Ann. Inst. H. Poincar\'{e} Anal. Non Lin\'{e}aire},
  35(2):443--480, 2018.

\bibitem{Aronson1967Local}
D.~G. Aronson and J.~Serrin.
\newblock Local behavior of solutions of quasilinear parabolic equations.
\newblock {\em Arch. Rational Mech. Anal.}, 25:81--122, 1967.

\bibitem{Aubin2009Set}
J.-P. Aubin and H.~Frankowska.
\newblock {\em Set-valued analysis}.
\newblock Modern Birkh\"auser Classics. Birkh\"auser Boston, Inc., Boston, MA,
  2009.
\newblock Reprint of the 1990 edition.

\bibitem{Benamou2017Variational}
J.-D. Benamou, G.~Carlier, and F.~Santambrogio.
\newblock Variational mean field games.
\newblock In {\em Active particles. {V}ol. 1. {A}d\-van\-ces in theory, models,
  and applications}, Model. Simul. Sci. Eng. Technol., pages 141--171.
  Birkh\"auser/Springer, Cham, 2017.

\bibitem{Cannarsa2004Semiconcave}
P.~Cannarsa and C.~Sinestrari.
\newblock {\em Semiconcave functions, {H}amilton--{J}acobi equations, and
  optimal control}.
\newblock Progress in Nonlinear Differential Equations and their Applications,
  58. Birkh\"auser Boston, Inc., Boston, MA, 2004.

\bibitem{CardaliaguetNotes}
P.~Cardaliaguet.
\newblock Notes on mean field games (from {P}.-{L}. {L}ions' lectures at
  {C}oll{\`e}ge de {F}rance).
\newblock Available at
  \url{https://www.ceremade.dauphine.fr/~cardaliaguet/MFG20130420.pdf}, 2013.

\bibitem{Cardaliaguet2013Long1}
P.~Cardaliaguet, J.-M. Lasry, P.-L. Lions, and A.~Porretta.
\newblock Long time average of mean field games with a nonlocal coupling.
\newblock {\em SIAM J. Control Optim.}, 51(5):3558--3591, 2013.

\bibitem{Evans1998Partial}
L.~C. Evans.
\newblock {\em Partial differential equations}, volume~19 of {\em Graduate
  Studies in Mathematics}.
\newblock American Mathematical Society, Providence, RI, 1998.

\bibitem{Feo2014Remark}
F.~Feo.
\newblock A remark on uniqueness of weak solutions for some classes of
  parabolic problems.
\newblock {\em Ric. Mat.}, 63(1, suppl.):S143--S155, 2014.

\bibitem{Fleming2006Controlled}
W.~H. Fleming and H.~M. Soner.
\newblock {\em Controlled {M}arkov processes and viscosity solutions},
  volume~25 of {\em Stochastic Modelling and Applied Probability}.
\newblock Springer, New York, second edition, 2006.

\bibitem{Gilbarg2001Elliptic}
D.~Gilbarg and N.~S. Trudinger.
\newblock {\em Elliptic partial differential equations of second order}.
\newblock Classics in Mathematics. Springer-Verlag, Berlin, 2001.
\newblock Reprint of the 1998 edition.

\bibitem{GomesVos}
D.~A. Gomes and V.~K. Voskanyan.
\newblock {Short-time existence of solutions for mean-field games with
  congestion}.
\newblock {\em Journal of the London Mathematical Society}, 92(3):778--799, 11
  2015.

\bibitem{Granas2003Fixed}
A.~Granas and J.~Dugundji.
\newblock {\em Fixed point theory}.
\newblock Springer Monographs in Mathematics. Springer-Verlag, New York, 2003.

\bibitem{Huang2003Individual}
M.~Huang, P.~E. Caines, and R.~P. Malham{\'e}.
\newblock Individual and mass behaviour in large population stochastic wireless
  power control problems: centralized and {N}ash equilibrium solutions.
\newblock In {\em 42nd IEEE Conference on Decision and Control, 2003.
  Proceedings}, volume~1, pages 98--103. IEEE, 2003.

\bibitem{Huang2007Large}
M.~Huang, P.~E. Caines, and R.~P. Malham\'e.
\newblock Large-population cost-coupled {LQG} problems with nonuniform agents:
  individual-mass behavior and decentralized {$\epsilon$}-{N}ash equilibria.
\newblock {\em IEEE Trans. Automat. Control}, 52(9):1560--1571, 2007.

\bibitem{Huang2006Large}
M.~Huang, R.~P. Malham\'e, and P.~E. Caines.
\newblock Large population stochastic dynamic games: closed-loop
  {M}c{K}ean--{V}lasov systems and the {N}ash certainty equivalence principle.
\newblock {\em Commun. Inf. Syst.}, 6(3):221--251, 2006.

\bibitem{Hughes2002Continuum}
R.~L. Hughes.
\newblock A continuum theory for the flow of pedestrians.
\newblock {\em Transportation Research Part B: Methodological}, 36(6):507--535,
  jul 2002.

\bibitem{Hughes2003Flow}
R.~L. Hughes.
\newblock The flow of human crowds.
\newblock In {\em Annual review of fluid mechanics, {V}ol. 35}, volume~35 of
  {\em Annu. Rev. Fluid Mech.}, pages 169--182. Annual Reviews, Palo Alto, CA,
  2003.

\bibitem{Ladyzhenskaja1968Linear}
O.~A. Lady\v{z}enskaja, V.~A. Solonnikov, and N.~N. Ural\cprime{}ceva.
\newblock {\em Linear and quasilinear equations of parabolic type}.
\newblock Translations of Mathematical Monographs, Vol. 23. American
  Mathematical Society, Providence, R.I., 1968.
\newblock Translated from the Russian by S. Smith.

\bibitem{Ladyzhenskaya1968Linear}
O.~A. Ladyzhenskaya and N.~N. Ural'tseva.
\newblock {\em Linear and quasilinear elliptic equations}.
\newblock Academic Press, New York-London, 1968.
\newblock Translated from the Russian by Scripta Technica, Inc, Translation
  editor: Leon Ehrenpreis.

\bibitem{Lasry2006JeuxI}
J.-M. Lasry and P.-L. Lions.
\newblock Jeux \`a champ moyen. {I}. {L}e cas stationnaire.
\newblock {\em C. R. Math. Acad. Sci. Paris}, 343(9):619--625, 2006.

\bibitem{Lasry2006JeuxII}
J.-M. Lasry and P.-L. Lions.
\newblock Jeux \`a champ moyen. {II}. {H}orizon fini et contr\^ole optimal.
\newblock {\em C. R. Math. Acad. Sci. Paris}, 343(10):679--684, 2006.

\bibitem{Lasry2007Mean}
J.-M. Lasry and P.-L. Lions.
\newblock Mean field games.
\newblock {\em Jpn. J. Math.}, 2(1):229--260, 2007.

\bibitem{Lieberman1996Second}
G.~M. Lieberman.
\newblock {\em Second order parabolic differential equations}.
\newblock World Scientific Publishing Co., Inc., River Edge, NJ, 1996.

\bibitem{Lions1980Resolution}
P.-L. Lions.
\newblock R\'{e}solution de probl\`emes elliptiques quasilin\'{e}aires.
\newblock {\em Arch. Rational Mech. Anal.}, 74(4):335--353, 1980.

\bibitem{LLperso}
P.-L. Lions.
\newblock Courses at {C}ollège de {France}, 2006--2012.
\newblock
  \url{http://www.college-de-france.fr/site/pierre-louis-lions/_course.htm}.

\bibitem{Mazanti2019Minimal}
G.~Mazanti and F.~Santambrogio.
\newblock Minimal-time mean field games.
\newblock {\em Math. Models Methods Appl. Sci.}, 29(8):1413--1464, 2019.

\bibitem{Porretta2015Weak}
A.~Porretta.
\newblock Weak solutions to {F}okker-{P}lanck equations and mean field games.
\newblock {\em Arch. Ration. Mech. Anal.}, 216(1):1--62, 2015.

\bibitem{Porzio1999Existence}
M.~M. Porzio.
\newblock Existence of solutions for some ``noncoercive'' parabolic equations.
\newblock {\em Discrete Contin. Dynam. Systems}, 5(3):553--568, 1999.

\bibitem{SantambrogioCIME}
F.~Santambrogio.
\newblock Lecture notes on variational mean field games.
\newblock Preprint cvgmt. \url{http://cvgmt.sns.it/paper/4646/}.

\bibitem{Santambrogio2015Optimal}
F.~Santambrogio.
\newblock {\em Optimal transport for applied mathematicians}, volume~87 of {\em
  Progress in Nonlinear Differential Equations and their Applications}.
\newblock Birkh\"auser/Springer, Cham, 2015.
\newblock Calculus of variations, PDEs, and modeling.

\bibitem{Simon1987Compact}
J.~Simon.
\newblock Compact sets in the space {$L^p(0,T;B)$}.
\newblock {\em Ann. Mat. Pura Appl. (4)}, 146:65--96, 1987.

\end{thebibliography}

\end{document}